\documentclass[12pt]{amsart}
\usepackage{amssymb,hyperref}

\makeatletter
\@namedef{subjclassname@2010}{%
  \textup{2010} Mathematics Subject Classification}
\makeatother

\newtheorem{theorem}{Theorem}[section]
\newtheorem{corollary}[theorem]{Corollary}
\newtheorem{lemma}[theorem]{Lemma}

\theoremstyle{definition}

\numberwithin{equation}{section}

\newcommand{\cA}{\mathcal{A}}
\newcommand{\cB}{\mathcal{B}}
\newcommand{\cD}{\mathcal{D}}
\newcommand{\cE}{\mathcal{E}}
\newcommand{\cP}{\mathcal{P}}

\newcommand{\N}{\mathbb{N}}
\newcommand{\R}{\mathbb{R}}
\newcommand{\Z}{\mathbb{Z}}

\newcommand{\eps}{\varepsilon}

\begin{document}
\baselineskip=17pt

\title[Hardy-Littlewood maximal operator]
{Hardy-Littlewood maximal operator\\
on reflexive variable Lebesgue spaces\\ 
over spaces of homogeneous type}

\author[A. Karlovich]{Alexei Karlovich}
\address{Centro de Matem\'atica e Aplica\c{c}\~oes,
                   Departamento de Matem\'a\-tica, 
                   Faculdade de Ci\^encias e Tecnologia,
                   Universidade Nova de Lisboa,
                   Quinta da Torre, 
                   2829--516 Caparica, Portugal}
\email{oyk@fct.unl.pt}


\date{}
\begin{abstract}
We show that the Hardy-Littlewood maximal operator is bounded on a reflexive 
variable Lebesgue space $L^{p(\cdot)}$ over a space of homogeneous type 
$(X,d,\mu)$ if and only if it is bounded on its dual space $L^{p'(\cdot)}$,
where $1/p(x)+1/p'(x)=1$ for $x\in X$. This result extends the corresponding 
result of Lars Diening from the Euclidean setting of $\mathbb{R}^n$ to the 
setting of spaces of homogeneous type $(X,d,\mu)$. 
\end{abstract}



\maketitle
\section{Introduction}
We begin with the definition of a space of homogeneous type (see, e.g., 
\cite{C90a}). Given a set $X$ and a function $d:X\times X\to[0,\infty)$, one 
says that $(X,d)$ is a quasi-metric space if the following axioms hold:
\begin{enumerate}
\item[(a)] $d(x,y)=0$ if and only if $x=y$;
\item[(b)] $d(x,y)=d(y,x)$ for all $x,y\in X$;
\item[(c)]  for all $x,y,z\in X$ and some 
constant $\kappa\ge 1$,
\begin{equation}\label{eq:quasi-triangle}
d(x,y)\le \kappa(d(x,y)+d(y,z)).
\end{equation}
\end{enumerate}
For $x\in X$ and $r>0$, consider the ball $B(x,r)=\{y\in X:d(x,y)<r\}$
centered at $x$ of radius $r$. Given a quasi-metric space $(X,d)$ and a 
positive measure $\mu$ that is defined on the $\sigma$-algebra generated by
quasi-metric balls, one says that $(X,d,\mu)$ is a space of homogeneous type
if there exists a constant $C_\mu\ge 1$ such that for any $x\in X$ and any 
$r>0$,
\begin{equation}\label{eq:doubling-measure}
\mu(B(x,2r))\le C_\mu \mu(B(x,r)).
\end{equation}
To avoid trivial measures, we will always assume that $0<\mu(B)<\infty$
for every ball $B$. Consequently, $\mu$ is a $\sigma$-finite measure.

Given a complex-valued function $f\in L^1_{\rm loc}(X,d,\mu)$, we define its 
Hardy-Littlewood maximal function $Mf$ by 
\[
(Mf)(x):=\sup_{B\ni x}\frac{1}{\mu(B)}\int_B |f(x)|\,d\mu(x),
\quad x\in X,
\]
where the supremum is taken over all balls $B\subset X$ containing $x\in X$.
The Hardy-Littlewood maximal operator $M$ is a sublinear operator acting by
the rule $f\mapsto Mf$.

Let $L^0(X,d,\mu)$ denote the set of all complex-valued measurable functions
on $X$ and let $\cP(X)$ denote the set of all measurable a.e. finite functions
$p:X\to[1,\infty]$. For a measurable set $E\subset X$, put
\[
p_-(E):=\operatornamewithlimits{ess\,inf}_{x\in E} p(x),
\quad
p_+(E):=\operatornamewithlimits{ess\,sup}_{x\in E} p(x)
\]
and
\[
p_-:=p_-(X),
\quad
p_+:=p_+(X).
\]
For a function $f\in L^0(X,d,\mu)$ and $p\in\cP(X)$, consider the functional,
which is called modular, given by
\[
\varrho_{p(\cdot)}(f):=\int_X |f(x)|^{p(x)}\,d\mu(x).
\]

By definition, the variable Lebesgue space $L^{p(\cdot)}(X,d,\mu)$ consists 
of all functions $f\in L^0(X,d,\mu)$ such that 
$\varrho_{p(\cdot)}(f/\lambda)<\infty$
for some $\lambda>0$ depending on $f$. It is a Banach space with respect
to the Luxemburg-Nakano norm given by
\[
\|f\|_{L^{p(\cdot)}}:=\inf\{\lambda>0\ :\ \varrho_{p(\cdot)}(f/\lambda)\le 1\}.
\]
If $p\in\cP(X)$ is constant, then $L^{p(\cdot)}(X,d,\mu)$ is nothing
but the standard Lebesgue space $L^p(X,d,\mu)$. Variable Lebesgue spaces 
are often called Naka\-no spaces. We refer to Maligranda's paper \cite{M11} 
for the role of Hidegoro Nakano in the study of variable Lebesgue spaces
and to the monographs \cite{CF13,DHHR11} for the basic properties of these
spaces. We only mention that the space $L^{p(\cdot)}(X,d,\mu)$
is reflexive if and only if $1<p_-,p_+<\infty$. In this case, the dual space
$[L^{p(\cdot)}(X,d,\mu)]^*$ is isomorphic to the space $L^{p'(\cdot)}(X,d,\mu)$,
where $p'\in\cP(X)$ is given by
\[
1/p(x)+1/p'(x)=1,\quad x\in X.
\]
(see, e.g., \cite[Propositin~2.79 and Corollary~2.81]{CF13}).

One of the central problems of Harmonic Analysis on variable Lebesgue spaces
is the problem of the boundedness of the Hardy-Littlewood maximal operator
$M$ on the space $L^{p(\cdot)}(X,d,\mu)$.  For a detailed history of this
problem, we refer to the monographs \cite{CF13,DHHR11,KMRS16}. 
We also mention that very recently Cruz-Uribe and Shukla 
\cite[Theorem~1.1]{CS18}  proved a sufficient condition for the boundedness 
of the fractional maximal operator $M_\alpha$, $0\le\alpha<1$, on reflexive
variable Lebesgue spaces  $L^{p(\cdot)}(X,d,\mu)$ over spaces of homogeneous
type, which includes the case of the Hardy-Littlewood maximal operator as a
 partial case when $\alpha=0$.

In 2005, Diening \cite[Theorem~8.1]{D05} (see also \cite[Theorem~5.7.2]{DHHR11})
among other things proved the following remarkable result: if $1<p_-(\R^n)$, 
$p_+(\R^n)<1$, then the Hardy-Littlewood maximal operator $M$ is bounded on 
the variable Lebesgue space $L^{p(\cdot)}(\R^n)$ if and only if it is bounded 
on its dual space $L^{p'(\cdot)}(\R^n)$. Recently Lerner 
\cite[Theorem~1.1]{L17} generalized this result to the setting of weighted
variable Lebesgue spaces $L_w^{p(\cdot)}(\R^n)$. The aim of this paper is 
to present a self-contained proof of the following extension of Diening's
theorem to the setting of spaces of homogeneous type.
\begin{theorem}[Main result]
\label{th:main}
Let $(X,d,\mu)$ be a space of homogeneous type and $p\in\cP(X)$ be such that
$1<p_-,p_+<\infty$. The Hardy-Littlewood maximal operator $M$ is bounded on 
the variable Lebesgue space $L^{p(\cdot)}(X,d,\mu)$ if and only if it 
is bounded on its dual space $L^{p'(\cdot)}(X,d,\mu)$.
\end{theorem}
Our approach is based on the adaptation of Lerner's proof \cite{L17}, which is
heavily based on the Calder\'on-Zygmund decomposition and dyadic maximal 
functions in the Euclidean setting of $\R^n$, to the setting of spaces of
homogeneous type. This becomes possible thanks to the recently developed
techniques of dyadic decomposition of spaces of homogeneous type due to 
Hyt\"onen and Kairema \cite{HK12} (see also previous works by Christ
\cite{C90a,C90b}). Note that this techniques was successfully applied in 
\cite{AHT17,AW18,CS18,K19} for studying various problems on spaces of 
homogeneous type, to mention a few recent works (this list is far from 
being exhaustive).

The paper is organized as follows. In Section~\ref{sec:DD-SHT}, we describe
the construction by Hyt\"onen and Kairema \cite{HK12} of a system of 
adjacent dyadic grids on a space of homogeneous type. Elements of this system
are called dyadic cubes, they have many important properties of usual dyadic
cubes of $\R^n$. 

In Section~\ref{sec:condition-A-infinity}, we recall the
definition of Banach function spaces and the main result of \cite{K19}
(see also \cite[Theorem~3.1]{L17})
saying that if the Hardy-Littlewood maximal operator $M$ is bounded on a 
Banach function space $\cE(X,d,\mu)$, then its boundedness on the associate
space $\cE'(X,d,\mu)$ is equivalent to a certain condition $\cA_\infty$.
Since the variable Lebesgue space $L^{p(\cdot)}(X,d,\mu)$ is a Banach function
space, in order to prove Theorem~\ref{th:main}, it is sufficient to verify that
$L^{p(\cdot)}(X,d,\mu)$ satisfies the condition $\cA_\infty$.

In Section~\ref{sec:modular-condition-A-infinity}, we recall very useful
relations between the norm and the modular in a varaible Lebesgue space. 
This allows us to formulate a modular analogue of the condition $\cA_\infty$ 
and show that this modular analogue implies the (norm) condition $\cA_\infty$.
The rest of the paper is devoted to the verification of the modular
analogue of the condition $\cA_\infty$ (see Lemma~\ref{le:2}).

In Section~\ref{sec:preparations}, we prepare the proof of the main result,
extending \cite[Lemmas~5.1--5.3 and 4.1]{L17} with $w\equiv 1$ from the 
Euclidean setting of $\R^n$ to the setting of spaces of homogeneous type.
Finally, in Section~\ref{sec:proof-main}, we complete the proof of
Theorem~\ref{th:main} following the scheme of the proof of 
\cite[Theorem~1.1]{L17}.
\section{Dyadic decomposition of spaces of homogeneous type.}
\label{sec:DD-SHT}
\subsection{Construction of Hyt\"onen and Kairema}
Let $(X,d,\mu)$ be a space of homogeneous type. The doubling property of $\mu$ 
implies the following geometric doubling property of the quasi-metric $d$:
any ball $B(x,r)$ can be covered by at most $N:=N(C_\mu,\kappa)$ balls of 
radius $r/2$. It is not difficult to show that $N\le C_\mu^{6+3\log_2\kappa}$.

An important tool for our proofs is the concepts of an adjacent system of 
dyadic grids $\cD^t$, $t\in\{1,\dots,K\}$, on  a space of homogeneous type 
$(X,d,\mu)$. Christ \cite[Theorem~11]{C90a} (see also 
\cite[Chap.~VI, Theorem~14]{C90b}) 
constructed a system of sets on  $(X,d,\mu)$, which satisfy many of the 
properties of a system of dyadic cubes on the Euclidean space. His 
construction was further refined by Hyt\"onen and 
Kairema  \cite[Theorem~2.2]{HK12}. We will use the version from 
\cite[Theorem~4.1]{AHT17}.
\begin{theorem}\label{th:Hytonen-Kairema}
Let $(X,d,\mu)$ be a space of homogeneous type with the constant $\kappa\ge 1$
in inequality \eqref{eq:quasi-triangle} and the geometric doubling constant $N$.
Suppose the parameter $\delta\in(0,1)$ satisfies $96\kappa^2\delta\le 1$. Then 
there exist an integer number $K=K(\kappa,N,\delta)$, a countable set of points
$\{z_\alpha^{k,t}:\alpha\in\cA_k\}$ with $k\in\Z$ and  $t\in\{1,\dots,K\}$, 
and a finite number of dyadic grids 
\[
\cD^t:=\{Q_\alpha^{k,t}:k\in\Z,\alpha\in\cA_k\}, 
\]
such that the following properties are fulfilled:
\begin{enumerate}
\item[{\rm(a)}]
for every $t\in\{1,\dots,K\}$ and $k\in\Z$ one has
\begin{enumerate}
\item[{\rm (i)}]
$X=\bigcup_{\alpha\in\cA_k} Q_\alpha^{k,t}$ (disjoint union);

\item[{\rm (ii)}]
if $Q,P\in\cD^t$, then $Q\cap P\in\{\emptyset, Q,P\}$;

\item[{\rm (iii)}]
if $Q_\alpha^{k,t}\in\cD^t$, then
\begin{equation}\label{eq:dyadic-cubes}
B(z_\alpha^{k,t},c_1\delta^k)
\subset 
Q_\alpha^{k,t}
\subset 
B(z_\alpha^{k,t},C_1\delta^k),
\end{equation}
where $c_1=(12\kappa^4)^{-1}$ and $C_1:=4\kappa^2$;
\end{enumerate}

\item[{\rm(b)}]
for every $t\in\{1,\dots,K\}$ and every $k\in\Z$, if $Q_\alpha^{k,t}\in\cD^t$, 
then there exists at least one $Q_\beta^{k+1,t}\in\cD^t$, which is called a 
child of $Q_\alpha^{k,t}$, such that  $Q_\beta^{k+1,t}\subset Q_\alpha^{k,t}$, 
and there exists exactly one $Q_\gamma^{k-1,t}\in\cD^t$, which is 
called the parent of $Q_\alpha^{k,t}$, such that 
$Q_\alpha^{k,t}\subset Q_\gamma^{k-1,t}$;

\item[{\rm (c)}]
for every ball $B=B(x,r)$, there exists 
\[
Q_B\in\bigcup_{t=1}^K\cD^t
\]
such that $B\subset Q_B$ and $Q_B=Q_\alpha^{k-1,t}$ for some indices 
$\alpha\in\cA_k$ and $t\in\{1,\dots,K\}$, where $k$ is 
the unique integer such that 
\[
\delta^{k+1}<r\le\delta^k.
\]
\end{enumerate}
\end{theorem}
The collections $\cD^t$, $t\in\{1,\dots,K\}$, are called dyadic grids on $X$. 
The sets $Q_\alpha^{k,t}\in\cD^t$ are referred to as dyadic cubes with center 
$z_\alpha^{k,t}$ and sidelength $\delta^k$, see \eqref{eq:dyadic-cubes}.
The sidelength of a cube $Q\in\cD^t$ will be denoted by $\ell(Q)$.
We should emphasize that these sets are not cubes in the standard sense even 
if the underlying space is $\R^n$. Parts (a) and (b) of the above theorem  
describe dyadic grids $\cD^t$, with $t\in\{1,\dots,K\}$, individually. 
In particular, 
\eqref{eq:dyadic-cubes} permits a comparison between a dyadic cube and
quasi-metric balls. Part (c) guarantees the existence of a finite family
of dyadic grids such that an arbitrary quasi-metric ball is contained in a 
dyadic cube in one of these grids. Such a finite family of dyadic grids is 
referred to as an adjacent system of dyadic grids. 
\subsection{Dyadic maximal function}
Let $\cD\in\bigcup_{t=1}^K \cD^t$ be a fixed dyadic grid. One can define the 
dyadic maximal function $M^\cD f$ of a function $f\in L_{\rm loc}^1(X,d,\mu)$
by
\[
(M^\cD f)(x)=\sup_{Q\ni x}\frac{1}{\mu(Q)}\int_Q |f(x)|\,d\mu(x),
\quad x\in X,
\]
where the supremum is taken over all dyadic cubes $Q\in\cD$ containing $x$.

The following important theorem is proved by Hyt\"onen and Kairema 
\cite[Proposition~7.9]{HK12}.
\begin{theorem}\label{th:HK-pointwise}
Let $(X,d,\mu)$ be a space of homogeneous type and let $\bigcup_{t=1}^K \cD^t$ 
be the adjacent system of dyadic grids given by Theorem~\ref{th:Hytonen-Kairema}.
There exist a constant $C_{HK}(X)\ge 1$ depending only $(X,d,\mu)$ such that 
for every function $f\in L_{\rm loc}^1(X,d,\mu)$ and a.e. $x\in X$, one has
\begin{align*}
(M^{\cD^t} f)(x)
&\le 
C_{HK}(X) (Mf)(x),\quad t\in\{1,\dots,K\},
\\
(Mf)(x)
&\le 
C_{HK}(X)\sum_{t=1}^K (M^{\cD^t}f)(x).
\end{align*}
\end{theorem}
\subsection{Calder\'on-Zygmund decomposition of a cube}
The following result is a consequence of Theorem~\ref{th:Hytonen-Kairema}.
\begin{lemma}\label{le:measure-child-father}
Suppose $(X,d,\mu)$ is a space of homogeneous type with the constants
$\kappa\ge 1$ in inequality \eqref{eq:quasi-triangle} and $C_\mu\ge 1$ in 
inequality \eqref{eq:doubling-measure}. Let $(X,d,\mu)$ be equipped with 
an adjacent system of dyadic grids $\{\cD^t,t=1,\dots,K\}$ and let 
$\delta\in(0,1)$ be chosen as in Theorem~\ref{th:Hytonen-Kairema}. Then there 
is an $\eps=\eps(\kappa,C_\mu,\delta)\in(0,1)$ such that for every 
$t\in\{1,\dots,K\}$ and all $Q,P\in\cD^t$, if $Q$ is a child of $P$, then 
\[
\mu(Q)\ge\eps\mu(P).
\]
\end{lemma}
We refer to \cite[Corollary~2.9]{AW18} or \cite[Lemma~8]{K19} for its proof.

Let $(X,d,\mu)$ be a space of homogeneous type and 
$\cD=\cD^{t_0}\in\bigcup_{t=1}^K\cD^t$ be a dyadic grid. Fix $Q_0\in\cD$.
Then there exists $k_0\in\Z$ and $\alpha_0\in\cA_{k_0}$ such that
$Q_0=Q_{\alpha_0}^{k_0,t_0}$. Consider
\begin{align}\label{eq:descendants-of-cube}
\cD(Q_0)
&:=
\{Q_\alpha^{k,t_0}\ :\ k\in\Z, \ k\ge k_0,\ \alpha\in\cA_k\}
\\
&=
\{Q'\in\cD^{t_0}\ :\ Q'\subset Q\},
\nonumber
\end{align}
that is, the set of all dyadic cubes with respect to $Q_0$. The set $\cD(Q_0)$ 
is formed by all dyadic descendants of the cube $Q_0$. For a measurable 
function $f$ such that
\begin{equation}\label{eq:L1-on-cube}
\int_{Q_0}|f(x)|\,d\mu(x)<\infty,
\end{equation}
consider the local dyadic maximal function of $f$ defined by
\[
(M^{\cD(Q_0)}f)(x):=\sup_{Q\ni x,\, Q\in\cD(Q_0)}
\frac{1}{\mu(Q)}\int_Q |f(x)|\,d\mu(x),
\quad x\in Q_0.
\]

Given a dyadic grid  $\cD\in\bigcup_{t=1}^K \cD^t$, a sparse family 
$S\subset\cD$ is a collection  of dyadic cubes $Q\in\cD$ for which there 
exists a collection of sets $\{E(Q)\}_{Q\in S}$ such that the sets $E(Q)$ are 
pairwise disjoint, $E(Q)\subset Q$, and 
\[
\mu(Q)\le 2\mu(E(Q)).
\]

We will need the following variation of the Calder\'on-Zygmund decomposition 
of the cube $Q_0$ (cf. \cite[Lemma~2.4]{L17}).
\begin{theorem}\label{th:local-CZ}
Let $(X,d,\mu)$ be a space of homogeneous type, $\cD\in\bigcup_{t=1}^K\cD^t$
be a dyadic grid, $Q_0\in\cD$, and $\cD(Q_0)$ be defined by 
\eqref{eq:descendants-of-cube}. Suppose $\eps\in(0,1)$ is the same as in
Lemma~\ref{le:measure-child-father}. For a nonzero measurable function $f$ on $Q_0$
satisfying \eqref{eq:L1-on-cube} and $k\in\N$, consider the sets
\begin{equation}\label{eq:local-CZ-1}
\Omega_k(Q_0):=\left\{
x\in Q_0\ :\ 
(M^{\cD(Q_0)}f)(x)>\left(\frac{2}{\eps}\right)^k
\frac{1}{\mu(Q_0)}\int_{Q_0}|f(x)|\,d\mu(x)
\right\}.
\end{equation}
If $\Omega_k(Q_0)\ne\emptyset$, then there exists a collection 
$\{Q_j^k(Q_0)\}_{j\in J_k}\subset\cD(Q_0)$ that is pairwise disjoint, maximal
with respect to inclusion, and such that
\begin{equation}\label{eq:local-CZ-2}
\Omega_k(Q_0)=\bigcup_{j\in J_k}Q_j^k(Q_0).
\end{equation}
The collection of cubes
\[
S:=\{Q_j^k(Q_0)\ :\ \Omega_k(Q_0)\ne\emptyset,\ j\in J_k\}
\]
is sparse, and for all $j$ and $k$, the sets
\[
E(Q_j^k(Q_0)):=Q_j^k(Q_0)\setminus\Omega_{k+1}(Q_0)
\]
satisfy
\begin{equation}\label{eq:local-CZ-3}
\mu\big(Q_j^k(Q_0)\big)\le 2\mu\big(E(Q_j^k(Q_0))\big).
\end{equation}
\end{theorem}
\begin{proof}
For each $k\in\N$ satisfying $\Omega_k(Q_0)\ne\emptyset$, the existence of a 
pairwise disjoint and maximal with respect to inclusion collection 
$\{Q_j^k(Q_0)\}_{j\in J_k}$, such that
\eqref{eq:local-CZ-2} is fulfilled, follows from \cite[Theorem~9(a)]{K19}.
Moreover, in view of the same theorem, for every $k\in\N$ 
satisfying $\Omega_k(Q_0)\ne\emptyset$ and $j\in J_k$, 
one has
\begin{align}\label{eq:local-CZ-4}
\left(\frac{2}{\eps}\right)^k\frac{1}{\mu(Q_0)}
\int_{Q_0}|f(x)|\,d\mu(x)
&<
\frac{1}{\mu(Q_j^k(Q_0))}\int_{Q_j^k(Q_0)}|f(x)|\,d\mu(x)
\\
&\le 
\left(\frac{2}{\eps}\right)^k\frac{1}{\eps\mu(Q_0)}
\int_{Q_0}|f(x)|\,d\mu(x).
\nonumber
\end{align}
It remains to prove \eqref{eq:local-CZ-3}. Since 
$\Omega_{k+1}(Q_0)\subset\Omega_k(Q_0)$ and, for each fixed $k$, the cubes 
$Q_j^k(Q_0)$ are pairwise disjoint, it is clear that the sets
$E(Q_j^k(Q_0))$ are pairwise disjoint for all $j$ and $k$. If
$Q_j^k(Q_0)\cap Q_i^{k+1}(Q_0)\ne \emptyset$, then by the maximality
of the cubes in $\{Q_j^k(Q_0)\}_{j\in J_k}$ and the fact that $2/\eps>1$,
we have $Q_i^{k+1}(Q_0)\subsetneqq Q_j^k(Q_0)$. In view of 
\eqref{eq:local-CZ-2} and \eqref{eq:local-CZ-4}, we see that
\begin{align*}
\mu\big(Q_j^k(Q_0)\cap\Omega_{k+1}(Q_0)\big)
&=
\sum_{\{i\,:\, Q_i^{k+1}(Q_0)\subsetneqq Q_j^k(Q_0)\}} 
\mu\big(Q_i^{k+1}(Q_0)\big)
\\
&\le 
\sum_{\{i\,:\, Q_i^{k+1}(Q_0)\subsetneqq Q_j^k(Q_0)\}}
\frac{\left(\frac{\eps}{2}\right)^{k+1}\int_{Q_i^{k+1}(Q_0)}|f(x)|\,d\mu(x)}
{\frac{1}{\mu(Q_0)}\int_{Q_0}|f(x)|\,d\mu(x)}
\\
&\le
\frac{\left(\frac{\eps}{2}\right)^{k+1}\int_{Q_j^k(Q_0)}|f(x)|\,d\mu(x)}
{\frac{1}{\mu(Q_0)}\int_{Q_0}|f(x)|\,d\mu(x)} 
\\
&\le 
\left(\frac{\eps}{2}\right)^{k+1}\left(\frac{2}{\eps}\right)^k
\frac{\mu(Q_j^k(Q_0))}{\eps}
=
\frac{\mu(Q_j^k(Q_0))}{2}.
\end{align*}
Then
\begin{align*}
\mu\big(E(Q_j^k(Q_0))\big)
&=
\mu\big(Q_j^k(Q_0)\setminus\Omega_{k+1}(Q_0)\big)
\\
&=\mu\big(Q_j^k(Q_0)\big)-\mu\big(Q_j^k (Q_0)\cap\Omega_{k+1}(Q_0)\big)
\\
&\ge 
\left(1-\frac{1}{2}\right)\mu\big(Q_j^k(Q_0)\big),
\end{align*}
whence $\mu\big(Q_j^k(Q_0)\big)\le 2\mu\big(E(Q_j^k(Q_0))\big)$ for all
$j$ and $k$, which completes the proof of \eqref{eq:local-CZ-3}.
\end{proof}
\section{Hardy-Littlewood maximal operator\\ on the associate space of a Banach 
function space}
\label{sec:condition-A-infinity}
\subsection{Banach function spaces}
Let us recall the definition of a Banach function space (see, e.g., 
\cite[Chap.~1, Definition~1.1]{BS88}).
Let $L_+^0(X,d,\mu)$ be the set of all non-negative measurable 
functions on $X$. The characteristic function of a set $E\subset X$ is denoted 
by $\chi_E$. A mapping $\rho: L_+^0(X,d,\mu)\to [0,\infty]$ is called 
a Banach function norm if, for all functions $f,g, f_n$
in the set $L_+^0(X,d,\mu)$ 
with $n\in\N$, for all constants $a\ge 0$, and for all measurable subsets 
$E$ of $X$, the following  properties hold:
\begin{eqnarray*}
{\rm (A1)} & &
\rho(f)=0  \Leftrightarrow  f=0\ \mbox{a.e.},
\
\rho(af)=a\rho(f),
\
\rho(f+g) \le \rho(f)+\rho(g),\\
{\rm (A2)} & &0\le g \le f \ \mbox{a.e.} \ \Rightarrow \ 
\rho(g) \le \rho(f)
\quad\mbox{(the lattice property)},\\
{\rm (A3)} & &0\le f_n \uparrow f \ \mbox{a.e.} \ \Rightarrow \
       \rho(f_n) \uparrow \rho(f)\quad\mbox{(the Fatou property)},\\
{\rm (A4)} & & \mu(E)<\infty\ \Rightarrow\ \rho(\chi_E) <\infty,\\
{\rm (A5)} & &\int_E f(x)\,d\mu(x) \le C_E\rho(f)
\end{eqnarray*}
with a constant $C_E \in (0,\infty)$ that may depend on $E$ and 
$\rho$,  but is independent of $f$. When functions differing only on 
a set of measure  zero are identified, the set $\cE(X,d,\mu)$ of all 
functions $f\in L^0(X,d,\mu)$ for  which  $\rho(|f|)<\infty$ is called a 
Banach function space. For each $f\in \cE(X,d,\mu)$, the norm of $f$ is 
defined by 
\[
\|f\|_\cE :=\rho(|f|). 
\]
The set $\cE(X,d,\mu)$ under the natural 
linear space operations and under  this norm becomes a Banach space (see 
\cite[Chap.~1, Theorems~1.4 and~1.6]{BS88}). 
If $\rho$ is a Banach function norm, its associate norm 
$\rho'$ is defined on $L_+^0(X,d,\mu)$ by
\[
\rho'(g):=\sup\left\{
\int_X f(x)g(x)\,d\mu(x) \ : \ 
f\in L_+^0(X,d,\mu), \ \rho(f) \le 1
\right\}.
\]
It is a Banach function norm itself \cite[Chap.~1, Theorem~2.2]{BS88}.
The Banach function space $\cE'(X,d,\mu)$ determined by the Banach function 
norm $\rho'$ is called the associate space (K\"othe dual) of $\cE(X,d,\mu)$.
\subsection{The condition $\cA_\infty$}
Following \cite{L17} and \cite[Definition~1]{K19},
we say that a Banach function space $\cE(X,d,\mu)$ over a space of homogeneous 
type $(X,d,\mu)$ satisfies the condition $\cA_\infty$ if there exist constants 
$\Phi,\theta>0$ such that for every dyadic grid $\cD\in\bigcup_{t=1}^K\cD^t$,
every finite sparse family $S\subset\cD$, every collection of non-negative
numbers $\{\alpha_Q\}_{Q\in S}$, and every collection of pairwise disjoint
measurable sets $\{G_Q\}_{Q\in S}$ such that $G_Q\subset Q$, one has
\begin{equation}\label{eq:A-infty}
\Bigg\|\sum_{Q\in S}\alpha_Q\chi_{G_Q}\Bigg\|_\cE
\le 
\Phi\left(\max_{Q\in S}\frac{\mu(G_Q)}{\mu(Q)}\right)^\theta
\Bigg\|\sum_{Q\in S}\alpha_Q\chi_Q\Bigg\|_\cE.
\end{equation}

The following result is a generalization of \cite[Theorem~3.1]{L17} from the 
Euclidean setting of $\R^n$ to the setting of spaces of homogeneous type.
\begin{theorem}[{\cite[Theorem~2]{K19}}]
\label{th:RAE}
Let $\cE(X,d,\mu)$ be a Banach function space over a space of homogeneous type 
$(X,d,\mu)$ and let $\cE'(X,d,\mu)$ be its associate space. 
\begin{enumerate}
\item[{\rm(a)}]
If the Hardy-Littlewood maximal operator $M$ is bounded on the space
$\cE'(X,d,\mu)$, then the space $\cE(X,d,\mu)$ satisfies the condition 
$\cA_\infty$.

\item[{\rm(b)}]
If the Hardy-Littlewood maximal operator $M$ is bounded on the space 
$\cE(X,d,\mu)$ and the space $\cE(X,d,\mu)$ satisfies the condition 
$\cA_\infty$, then the operator $M$ is bounded on the space $\cE'(X,d,\mu)$.
\end{enumerate}
\end{theorem} 
Since the variable Lebesgue space $L^{p(\cdot)}(X,d,\mu)$
is a Banach function space and, under the condition $1<p_-$, $p_+<\infty$,
its associate space $[L^{p(\cdot)}(X,d,\mu)]'$ is isomorphic to the variable
Lebesgue space $L^{p'(\cdot)}(X,d,\mu)$ (see, e.g., 
\cite[Section~2.10.3]{CF13}), Theorem~\ref{th:RAE}(b) immediately implies 
the following.
\begin{corollary}\label{co:RAE}
Let $(X,d,\mu)$ be a space of homogeneous type and $p\in\cP(X)$ be such that
$1<p_-$, $p_+<\infty$. If the Hardy-Littlewood maximal operator $M$ is bounded
on the variable Lebesgue space $L^{p(\cdot)}(X,d,\mu)$ and the space
$L^{p(\cdot)}(X,d,\mu)$ satisfies the condition $\cA_\infty$, then $M$
is bounded on the dual space $L^{p'(\cdot)}(X,d,\mu)$.
\end{corollary}
It follows from Corollary~\ref{co:RAE} that in order to prove 
Theorem~\ref{th:main}, it is sufficient to verify that the space 
$L^{p(\cdot)}(X,d,\mu)$ satisfies the condition $\cA_\infty$. 
\section{Norm inequalities and modular inequalities}
\label{sec:modular-condition-A-infinity}
\subsection{Norm-modular unit ball property}
In this subsection we formulate two very useful properties that relate norms
and modulars in variable Lebesgue spaces.
\begin{lemma}[{see, e.g., \cite[Lemma~3.2.4]{DHHR11}}]
\label{le:3.2.4}
Let $(X,d,\mu)$ be a space of homogeneous type and $p\in\cP(X)$. Then 
for every $f\in L^{p(\cdot)}(X,d,\mu)$ the inequalities 
$\|f\|_{L^p(\cdot)}\le 1$ and $\varrho_{p(\cdot)}(f)\le 1$ are equivalent.
\end{lemma}
\begin{lemma}[{see, e.g., \cite[Lemma~3.2.5]{DHHR11}}]
\label{le:3.2.5}
Let $(X,d,\mu)$ be a space of homogeneous type and $p\in\cP(X)$ be such that
$1<p_-$, $p_+<\infty$. Then for every $f\in L^{p(\cdot)}(X,d,\mu)$,
\[
\min\left\{
\varrho_{p(\cdot)}(f)^{\frac{1}{p_-}},
\varrho_{p(\cdot)}(f)^{\frac{1}{p_+}}
\right\}
\le 
\|f\|_{L^{p(\cdot)}}
\le 
\max\left\{
\varrho_{p(\cdot)}(f)^{\frac{1}{p_-}},
\varrho_{p(\cdot)}(f)^{\frac{1}{p_+}}
\right\}.
\]
\end{lemma}
\subsection{Auxiliary lemma}
The following auxiliary lemma illustrates the possibility of substitution
of norm inequalities by modular inequalities.
\begin{lemma}\label{le:1}
Let $(X,d,\mu)$ be a space of homogeneous type and $p\in\cP(X)$ satisfy
$1<p_-,p_+<\infty$. Suppose $\cD\in\bigcup_{t=1}^K\cD^t$ is a dyadic
grid. If $S\in\cD$ is a finite family and $\{\alpha_Q\}_{Q\in S}$ is
a family of nonnegative numbers such that
\[
\left\|\sum_{Q\in S}\alpha_Q\chi_Q\right\|_{L^{p(\cdot)}}\le 1,
\]
then
\[
\sum_{Q\in S}\int_Q\alpha_Q^{p(x)}\,\mu(x)\le 1.
\]
\end{lemma}
\begin{proof}
It is clear that
\begin{equation}\label{eq:1-1}
\sum_{Q\in S}\int_Q\alpha_Q^{p(x)}\,d\mu(x)
=
\int_X\left(\sum_{Q\in S}\left(\alpha_Q\chi_Q(x)\right)^{p(x)}\right)\,d\mu(x).
\end{equation}
Since $1<p_-\le p(x)\le p_+<\infty$ for a.e. $x\in X$,  one has
\begin{equation}\label{eq:1-2}
\sum_{Q\in S}\left(\alpha_Q\chi_Q(x)\right)^{p(x)}
\le 
\left(\sum_{Q\in S}\alpha_Q\chi_Q(x)\right)^{p(x)}.
\end{equation}
Taking into account \eqref{eq:1-1} and \eqref{eq:1-2}, it follows from
Lemma~\ref{le:3.2.4} that
\[
\sum_{Q\in S}\int_Q\alpha_Q^{p(x)}\,d\mu(x) 
\le 
\int_X\left(\sum_{Q\in S}\alpha_Q\chi_Q(x)\right)^{p(x)}\,d\mu(x)\le 1,
\]
which completes the proof.
\end{proof}
\subsection{Modular version of the condition $\cA_\infty$}
In this subsection we formulate a modular analogue of the condition 
$\cA_\infty$ and show that it implies the (norm) condition $\cA_\infty$.
\begin{lemma}\label{le:2}
Let $(X,d,\mu)$ be a space of homogeneous type and $p\in\cP(X)$ satisfy
$1<p_-,p_+<\infty$. If there exist constants $\Psi,\xi>1$
such that for every dyadic grid $\cD\in\bigcup_{t=1}^K\cD^t$,
every finite sparse family $S\subset\cD$, every collection of
pairwise disjoint measurable sets $\{G_Q\}_{Q\in S}$ such that
$G_Q\subset Q$ and every collection of nonnegative numbers 
$\{\alpha_Q\}_{Q\in S}$ such that
\begin{equation}\label{eq:2-1}
\left\|\sum_{Q\in S}\alpha_Q\chi_Q\right\|_{L^{p(\cdot)}}=1,
\end{equation}
one has
\begin{equation}\label{eq:2-2}
\sum_{Q\in S}\int_{G_Q}\alpha_Q^{p(x)}\,d\mu(x)
\le 
\Psi\left(\max_{Q\in S}\frac{\mu(G_Q)}{\mu(Q)}\right)^\xi,
\end{equation}
then the variable Lebesgue space $L^{p(\cdot)}(X,d,\mu)$
satisfies the condition $\cA_\infty$.
\end{lemma}
\begin{proof}
Fix a dyadic grid $\cD\in\bigcup_{t=1}^K\cD^t$, a finite sparse family
$S\subset\cD$, a collection of pairwise disjoint measurable sets 
$\{G_Q\}_{Q\in S}$ such that $G_Q\subset Q$. Let $\{\beta_Q\}_{Q\in S}$
be an arbitrary collection of nonnegative numbers. Put
\begin{equation}\label{eq:2-3}
\alpha_Q:=\frac{\beta_Q}
{\displaystyle\left\|\sum_{Q\in S}\beta_Q\chi_Q\right\|_{L^{p(\cdot)}}}.
\end{equation}
Then \eqref{eq:2-1} is fulfilled. since the sets $\{G_Q\}_{Q\in S}$ are
pairwise disjoint, we have
\[
\left(\sum_{Q\in S}\alpha_Q\chi_{G_Q}(x)\right)^{p(x)}
=
\sum_{Q\in S}\alpha_Q^{p(x)}\chi_{G_Q}(x),
\quad x\in X.
\]
Hence
\begin{equation}\label{eq:2-4}
\sigma
:=
\sum_{Q\in S}\int_{G_Q}\alpha_Q^{p(x)}\,d\mu(x)
=
\int_X\left(\sum_{Q\in S}\alpha_Q\chi_{G_Q}(x)\right)^{p(x)}\,d\mu(x).
\end{equation}
By Lemma~\ref{le:3.2.5} and \eqref{eq:2-2}, \eqref{eq:2-4}, we have
\begin{align}\label{eq:2-5}
&
\left\|\sum_{Q\in S}\alpha_Q\chi_{G_Q}\right\|_{L^{p(\cdot)}}
\le 
\max\left\{\sigma^{\frac{1}{p_-}},\sigma^{\frac{1}{p_+}}\right\}
\\
&\quad\quad\le
\max\left\{
\Psi^{\frac{1}{p_-}}
\left(\max_{Q\in S}\frac{\mu(G_Q)}{\mu(Q)}\right)^{\frac{\xi}{p_-}},
\Psi^{\frac{1}{p_+}}
\left(\max_{Q\in S}\frac{\mu(G_Q)}{\mu(Q)}\right)^{\frac{\xi}{p_+}}
\right\}
\nonumber\\
&\quad\quad\le
\Psi^{\frac{1}{p_-}}
\left(\max_{Q\in S}\frac{\mu(G_Q)}{\mu(Q)}\right)^{\frac{\xi}{p_+}}
\nonumber
\end{align}
because $\Psi>1$, $\mu(G_Q)\le\mu(Q)$ for all $Q\in S$ and $p_-\le p_+$.
It follows from \eqref{eq:2-3} and \eqref{eq:2-5} that
\[
\left\|\sum_{Q\in S}\beta_Q\chi_{G_Q}\right\|_{L^{p(\cdot)}}
\le 
\Psi^{\frac{1}{p_-}}
\left(\max_{Q\in S}\frac{\mu(G_Q)}{\mu(Q)}\right)^{\frac{\xi}{p_+}}
\left\|\sum_{Q\in S}\beta_Q\chi_Q\right\|_{L^{p(\cdot)}},
\]
that is, the space $L^{p(\cdot)}(X,d,\mu)$ satisfies the condition
$\cA_\infty$ with $\Phi=\Psi^{\frac{1}{p_-}}$ and $\theta=\xi/p_+$.
\end{proof}
\section{Preparations for the verification of the condition $\cA_\infty$}
\label{sec:preparations}
\subsection{First lemma}
Let $\|M\|_{\cB(L^{p(\cdot)})}$ denote the norm of the Hardy-Little\-wood 
maximal operator on the variable Lebesgue space $L^{p(\cdot)}(X,d,\mu)$.
As usual, for an exponent $r\in(1,\infty)$, let $r'=r/(r-1)\in(1,\infty)$ 
denote the conjugate exponent.

The preparations of the verification of the condition $\cA_\infty$
in the proof of Theorem~\ref{th:main} consist of four steps.
The first step is the proof of the following extension of \cite[Lemma~5.1]{L17}
with $w\equiv 1$ from the Euclidean setting of $\R^n$  to the setting of 
spaces of homogeneous type.
\begin{lemma}\label{le:3}
Let $(X,d,\mu)$ be a space of homogeneous type and $p\in\cP(X)$ satisfy
$1<p_-,p_+<\infty$. Suppose the Hardy-Littlewood maximal operator $M$
is bounded on $L^{p(\cdot)}(X,d,\mu)$. There exist constants $A,\lambda>1$ 
such that for every dyadic grid $\cD\in\bigcup_{t=1}^K\cD^t$,
every family of pairwise disjoint cubes $S_d\subset\cD$, every family
of nonnegative numbers $\{\alpha_Q\}_{Q\in S_d}$, if
\begin{equation}\label{eq:3-1}
\sum_{Q\in S_d}\int_Q\alpha_Q^{p(x)}\,d\mu(x)\le 1,
\end{equation}
then
\begin{equation}\label{eq:3-2}
\sum_{Q\in S_d}\mu(Q)\left(
\frac{1}{\mu(Q)}\int_Q\alpha_Q^{\lambda p(x)}\,d\mu(x)
\right)^{\frac{1}{\lambda}}\le A.
\end{equation}
\end{lemma}
\begin{proof}
Let $\eps\in(0,1)$ be the same as in Lemma~\ref{le:measure-child-father}. Fix a 
dyadic grid $\cD\in\bigcup_{t=1}^K\cD^t$ and a family of pairwise disjoint 
cubes $S_d\subset\cD$. For $k\in\N$ and $Q\in S_d$, put
\begin{equation}\label{eq:3-3}
\Omega_k(Q):=\left\{
x\in Q\ :\ (M^{\cD(Q)}\alpha_Q^{p(\cdot)})(x)>\left(\frac{2}{\eps}\right)^k
\frac{1}{\mu(Q)}\int_Q\alpha_Q^{p(x)}\,d\mu(x)
\right\}.
\end{equation}
By Theorem~\ref{th:local-CZ}, if these sets are nonempty, then they can be 
written as follows:
\begin{equation}\label{eq:3-4}
\Omega_k(Q)=\bigcup_j Q_j^k(Q),
\end{equation}
where
$Q_j^k(Q)\in\cD(Q)$ are pairwise disjoint cubes for all $j$ and $k$, and
\begin{equation}\label{eq:3-5}
\mu\big(Q_j^k(Q)\setminus\Omega_{k+1}(Q)\big)
\ge
\frac{1}{2}\mu\big(Q_j^k(Q)\big).
\end{equation}

Fix $k\in\N$ and $Q\in S_d$. If $x\in\Omega_k(Q)$, then in view of 
\eqref{eq:3-4} there exists $j_0$ such that $x\in Q_{j_0}^k(Q)$. It follows
from \eqref{eq:3-5} that
\begin{align*}
\chi_{\Omega_k(Q)}(x)
&\le 
\frac{2\mu\big(Q_{j_0}^k(Q)\setminus\Omega_{k+1}(Q)\big)}
{\mu\big(Q_{j_0}^k(Q)\big)}
\\
&=
\frac{2\mu\big(Q_{j_0}^k(Q)\cap(\Omega_k(Q)\setminus\Omega_{k+1}(Q))\big)}
{\mu\big(Q_{j_0}^k(Q)\big)}
\\
&=
\frac{2}{\mu\big(Q_{j_0}^k(Q)\big)}
\int_{Q_{j_0}^k(Q)}\chi_{\Omega_k(Q)\setminus\Omega_{k+1}(Q)}(y)\,d\mu(y)
\\
&\le 
2\sup_{Q'\ni x,\, Q'\in\cD}\frac{1}{\mu(Q')}
\int_{Q'}\chi_{\Omega_k(Q)\setminus\Omega_{k+1}(Q)}(y)\,d\mu(y),
\end{align*}
which implies that
\[
\chi_{\Omega_k(Q)}(x)
\le 
2(M^\cD\chi_{\Omega_k(Q)\setminus\Omega_{k+1}(Q)})(x),
\quad
x\in X.
\]
Therefore, for all $k\in\N$ and $Q\in S_d$,
\begin{align*}
\alpha_Q\chi_{\Omega_k(Q)}(x)
&\le 
2\big(M^\cD(\alpha_Q\chi_{\Omega_k(Q)\setminus\Omega_{k+1}(Q)})\big)(x)
\\
&\le 
2\left(M^\cD\left(
\sum_{Q\in S_d}\alpha_Q\chi_{\Omega_k(Q)\setminus\Omega_{k+1}(Q)}
\right)\right)(x),
\quad
x\in X.
\end{align*}
Since the cubes in the collection $S_d$ are pairwise disjoint, the sets 
in the collection  $\{\Omega_k(Q)\}_{Q\in S_d}$ are also pairwise disjoint 
for every fixed $k\in\N$. Hence, the above inequality implies that for 
$k\in\N$,
\begin{equation}\label{eq:3-6}
\sum_{Q\in S_d}\alpha_Q\chi_{\Omega_k(Q)}(x)
\le 
2\left(M^\cD\left(
\sum_{Q\in S_d}\alpha_Q\chi_{\Omega_k(Q)\setminus\Omega_{k+1}(Q)}
\right)\right)(x),
\quad
x\in X.
\end{equation}
It follows from the boundedness of $M$ on $L^{p(\cdot)}(X,d,\mu)$,
Theorem~\ref{th:HK-pointwise}, and inequality \eqref{eq:3-6} that
\begin{equation}\label{eq:3-7}
\left\|\sum_{Q\in S_d}\alpha_Q\chi_{\Omega_k(Q)}\right\|_{L^{p(\cdot)}}
\le 
2C_{HK}(X)\|M\|_{\cB(L^{p(\cdot)})}
\left\|
\sum_{Q\in S_d}\alpha_Q\chi_{\Omega_k(Q)\setminus\Omega_{k+1}(Q)}
\right\|_{L^{p(\cdot)}}.
\end{equation}
Set
\begin{equation}\label{eq:3-8}
\widetilde{\alpha}_Q:=\alpha_Q\left(
\left\|\sum_{Q\in S_d}\alpha_Q\chi_{\Omega_k(Q)}\right\|_{L^{p(\cdot)}}
\right)^{-1}.
\end{equation}
Then inequality \eqref{eq:3-7} can be rewritten as follows:
\begin{equation}\label{eq:3-9}
\frac{1}{2C_{HK}(X)\|M\|_{\cB(L^{p(\cdot)})}}
\le 
\left\|
\sum_{Q\in S_d}\widetilde{\alpha}_Q\chi_{\Omega_k(Q)\setminus\Omega_{k+1}(Q)}
\right\|_{L^{p(\cdot)}}.
\end{equation}
It follows from \eqref{eq:3-8} that
\begin{equation}\label{eq:3-10}
\left\|
\sum_{Q\in S_d}\widetilde{\alpha}_Q\chi_{\Omega_k(Q)\setminus\Omega_{k+1}(Q)}
\right\|_{L^{p(\cdot)}}
\le 
\left\|
\sum_{Q\in S_d}\widetilde{\alpha}_Q\chi_{\Omega_k(Q)}
\right\|_{L^{p(\cdot)}}
=1.
\end{equation}
Inequality \eqref{eq:3-10} and Lemma~\ref{le:3.2.5} imply that
\begin{align}\label{eq:3-11}
&
\left\|
\sum_{Q\in S_d}\widetilde{\alpha}_Q\chi_{\Omega_k(Q)\setminus\Omega_{k+1}(Q)}
\right\|_{L^{p(\cdot)}}
\\
&\quad\quad\le 
\left(\int_X
\left(
\sum_{Q\in S_d}\widetilde{\alpha}_Q\chi_{\Omega_k(Q)\setminus\Omega_{k+1}(Q)}(x)
\right)^{p(x)}
\,d\mu(x)\right)^{\frac{1}{p_+}}.
\nonumber
\end{align}
Since the cubes in $S_d$ are pairwise disjoint, we conclude that the sets
in the collection $\{\Omega_k(Q)\setminus\Omega_{k+1}(Q)\}_{Q\in S_d}$
are also pairwise disjoint. Therefore, we deduce from \eqref{eq:3-9} and 
\eqref{eq:3-11} that
\begin{align}\label{eq:3-12}
&
\left(\frac{1}{2C_{HK}(X)\|M\|_{\cB(L^{p(\cdot)})}}\right)^{p_+}
\\
&\quad\quad
\le 
\int_X\left(
\sum_{Q\in S_d}\widetilde{\alpha}_Q\chi_{\Omega_k(Q)\setminus\Omega_{k+1}(Q)}(x)
\right)^{p(x)}\,d\mu(x)
\nonumber\\
&\quad\quad
=
\sum_{Q\in S_d}
\int_{\Omega_k(Q)\setminus\Omega_{k+1}(Q)}(\widetilde{\alpha}_Q)^{p(x)}\,d\mu(x)
\nonumber\\
&\quad\quad
=
\sum_{Q\in S_d}
\int_{\Omega_k(Q)}(\widetilde{\alpha}_Q)^{p(x)}\,d\mu(x)
-
\sum_{Q\in S_d}
\int_{\Omega_{k+1}(Q)}(\widetilde{\alpha}_Q)^{p(x)}\,d\mu(x).
\nonumber
\end{align}
Again, taking into account that the cubes in $S_d$ are pairwise disjoint, we 
conclude from \eqref{eq:3-10} and Lemma~\ref{le:3.2.4} that
\begin{align}\label{eq:3-13}
\sum_{Q\in S_d}\int_{\Omega_k(Q)}(\widetilde{\alpha}_Q)^{p(x)}\,d\mu(x)
&=
\int_X\left(
\sum_{Q\in S_d}\widetilde{\alpha}_Q\chi_{\Omega_k(Q)}(x)
\right)^{p(x)}\,d\mu(x)
\\
&\le
\left\|\sum_{Q\in S_d}\widetilde{\alpha}_Q\chi_{\Omega_k(Q)}\right\|_{L^{p(\cdot)}}
=1.
\nonumber
\end{align}
Since $\Omega_{k+1}(Q)\subset\Omega_k(Q)$, it follows from \eqref{eq:3-10}
that
\[
\left\|\sum_{Q\in S_d}\widetilde{\alpha}_Q\chi_{\Omega_{k+1}(Q)}\right\|_{L^{p(\cdot)}}
\le 
\left\|\sum_{Q\in S_d}\widetilde{\alpha}_Q\chi_{\Omega_k(Q)}\right\|_{L^{p(\cdot)}}
=1.
\]
Hence, in view of Lemma~\ref{le:3.2.5}, we have
\begin{align}\label{eq:3-14}
\left\|\sum_{Q\in S_d}\widetilde{\alpha}_Q\chi_{\Omega_{k+1}(Q)}\right\|_{L^{p(\cdot)}}
&\le 
\left(\int_X
\left(\sum_{Q\in S_d}\widetilde{\alpha}_Q\chi_{\Omega_{k+1}(Q)}(x)\right)^{p(x)}\,d\mu(x)
\right)^{\frac{1}{p_+}}
\\
&=
\left(
\sum_{Q\in S_d}\int_{\Omega_{k+1}(Q)}
(\widetilde{\alpha}_Q)^{p(x)}\,d\mu(x)
\right)^{\frac{1}{p_+}}.
\nonumber
\end{align}
It follows from \eqref{eq:3-12}--\eqref{eq:3-14} that
\[
\left\|\sum_{Q\in S_d}\widetilde{\alpha}_Q\chi_{\Omega_{k+1}(Q)}\right\|_{L^{p(\cdot)}}
\le 
\left\{
1-\left(\frac{1}{2C_{HK}(X)\|M\|_{\cB(L^{p(\cdot)})}}\right)^{p_+}
\right\}^{\frac{1}{p_+}}=:\beta.
\]
The above inequality and equality \eqref{eq:3-8} imply that for $k\in\N$,
\begin{equation}\label{eq:3-15}
\left\|\sum_{Q\in S_d}\alpha_Q\chi_{\Omega_{k+1}(Q)}\right\|_{L^{p(\cdot)}}
\le\beta
\left\|\sum_{Q\in S_d}\alpha_Q\chi_{\Omega_k(Q)}\right\|_{L^{p(\cdot)}}.
\end{equation}
It follows from \eqref{eq:3-1} and Lemma~\ref{le:3.2.4} that
\begin{equation}\label{eq:3-16}
\left\|\sum_{Q\in S_d}\alpha_Q\chi_Q\right\|_{L^{p(\cdot)}}\le 1.
\end{equation}
Since $\Omega_1(Q)\subset Q$, applying \eqref{eq:3-16} and then applying
\eqref{eq:3-15} $k-1$ times, we get
\begin{align*}
1 
&\ge
\left\|\sum_{Q\in S_d}\alpha_Q\chi_{\Omega_1(Q)}\right\|_{L^{p(\cdot)}}
\\
&\ge
\frac{1}{\beta}
\left\|\sum_{Q\in S_d}\alpha_Q\chi_{\Omega_2(Q)}\right\|_{L^{p(\cdot)}}
\\
&\ge \dots
\\
&\ge \frac{1}{\beta^{k-1}} 
\left\|\sum_{Q\in S_d}\alpha_Q\chi_{\Omega_k(Q)}\right\|_{L^{p(\cdot)}}.
\end{align*}
Thus
\[
\left\|\sum_{Q\in S_d}\alpha_Q\chi_{\Omega_k(Q)}\right\|_{L^{p(\cdot)}}
\le 
\beta^{k-1},
\quad k\in\N.
\]
In view of Lemma~\ref{le:3.2.5}, this inequality implies that
\begin{equation}\label{eq:3-17}
\sum_{Q\in S_d}\int_{\Omega_k(Q)}\alpha_Q^{p(x)}\,d\mu(x)
\le 
\beta^{p_-(k-1)},
\quad 
k\in\N.
\end{equation}

Fix $Q\in S_d$. Put $\Omega_0(Q):=Q$. Then it follows from \eqref{eq:3-3}
that for $k\in\Z_+$ and $x\in\Omega_k(Q)\setminus\Omega_{k+1}(Q)$, one has
\[
\alpha_Q^{p(x)}
\le 
(M^{\cD(Q)}\alpha_Q^{p(\cdot)})(x)
\le 
\left(\frac{2}{\eps}\right)^{k+1}
\frac{1}{\mu(Q)}\int_Q\alpha_Q^{p(y)}\,d\mu(y).
\]
Taking into account this inequality, we obtain for every $\phi>0$,
\begin{align}\label{eq:3-18}
\int_Q &\alpha_Q^{p(x)(1+\phi)}\,d\mu(x)
=
\sum_{k=0}^\infty\int_{\Omega_k(Q)\setminus\Omega_{k+1}(Q)}
\alpha_Q^{p(x)(1+\phi)}\,d\mu(x)
\\
&\quad\le 
\left(\frac{1}{\mu(Q)}\int_Q\alpha_Q^{p(y)}\,d\mu(y)\right)^\phi 
\sum_{k=0}^\infty\left(\frac{2}{\eps}\right)^{\phi(k+1)}
\int_{\Omega_k(Q)\setminus\Omega_{k+1}(Q)}\alpha_Q^{p(x)}\,d\mu(x)
\nonumber\\
&\quad\le 
\left(\frac{1}{\mu(Q)}\int_Q\alpha_Q^{p(y)}\,d\mu(y)\right)^\phi 
\sum_{k=0}^\infty\left(\frac{2}{\eps}\right)^{\phi(k+1)}
\int_{\Omega_k(Q)}\alpha_Q^{p(x)}\,d\mu(x).
\nonumber
\end{align}
It is easy to see that one can choose $\phi>0$ such that
\[
0<\left(\frac{2}{\eps}\right)^\phi\beta^{p_-}<1.
\]
Then
\begin{equation}\label{eq:3-19}
\sum_{k=0}^\infty
\left[\left(\frac{2}{\eps}\right)^\phi\beta^{p_-}\right]^{k+1}<\infty.
\end{equation}
Take $\lambda:=1+\phi$. By \eqref{eq:3-18}, we have
\begin{align}\label{eq:3-20}
\sum_{Q\in S_d}
& \mu(Q)\left(\frac{1}{\mu(Q)}
\int_Q\alpha_Q^{\lambda p(x)}\,d\mu(x)\right)^{\frac{1}{\lambda}}
\\
&\le 
\sum_{Q\in S_d}\mu(Q)
\left(
\frac{1}{\mu(Q)}
\left(\frac{1}{\mu(Q)}\int_Q\alpha_Q^{p(y)}\,d\mu(y)\right)^\phi
\right)^{\frac{1}{\lambda}}
\nonumber\\
&\quad\quad\quad\times
\left(\sum_{k=0}^\infty\left(\frac{2}{\eps}\right)^{\phi(k+1)}
\int_{\Omega_k(Q)}\alpha_Q^{p(x)}\,d\mu(x)
\right)^{\frac{1}{\lambda}}.
\nonumber
\end{align}
Since $\frac{1}{\lambda}=\frac{1}{1+\phi}$ and 
$\frac{1}{\lambda'}=\frac{\phi}{1+\phi}$, we have
\begin{align}\label{eq:3-21}
&
\mu(Q)\left(\frac{1}{\mu(Q)}
\left(\frac{1}{\mu(Q)}\int_Q\alpha_Q^{p(y)}\,d\mu(y)\right)^\phi
\right)^{\frac{1}{\lambda}}
\\
&\quad\quad=
\mu(Q)\left(\frac{1}{(\mu(Q))^{1+\phi}}\right)^{\frac{1}{\lambda}}
\left(
\int_Q\alpha_Q^{p(x)}\,d\mu(x)
\right)^{\frac{\phi}{\lambda}}
\nonumber\\
&\quad\quad=
\left(\int_Q\alpha_Q^{p(x)}\,d\mu(x)\right)^{\frac{1}{\lambda'}}.
\nonumber
\end{align}
Combining inequality \eqref{eq:3-20} with equality \eqref{eq:3-21}, applying
H\"older's inequality, and taking into account inequality \eqref{eq:3-1},
we obtain
\begin{align}\label{eq:3-22}
&
\sum_{Q\in S_d}\mu(Q)
\left(\frac{1}{\mu(Q)}\int_Q\alpha_Q^{\lambda p(x)}\,d\mu(x)\right)^{\frac{1}{\lambda}}
\\
&\quad\le
\sum_{Q\in S_d}
\left(\int_Q\alpha_Q^{p(x)}\,d\mu(x)\right)^{\frac{1}{\lambda'}}
\left(\sum_{k=0}^\infty\left(\frac{2}{\eps}\right)^{\phi(k+1)}
\int_{\Omega_k(Q)}\alpha_Q^{p(x)}\,d\mu(x)\right)^{\frac{1}{\lambda}}
\nonumber\\
&\quad\le 
\left(\sum_{Q\in S_d}\int_Q\alpha_Q^{p(x)}\,d\mu(x)\right)^{\frac{1}{\lambda'}}
\left(\sum_{Q\in S_d}\sum_{k=0}^\infty\left(\frac{2}{\eps}\right)^{\phi(k+1)}
\int_{\Omega_k(Q)}\alpha_Q^{p(x)}\,d\mu(x)\right)^{\frac{1}{\lambda}}
\nonumber\\
&\quad=
\left(\sum_{k=0}^\infty\left(\frac{2}{\eps}\right)^{\phi(k+1)}
\sum_{Q\in S_d}\int_{\Omega_k(Q)}\alpha_Q^{p(x)}\,d\mu(x)
\right)^{\frac{1}{\lambda}}.
\nonumber
\end{align}
It follows from \eqref{eq:3-1}, \eqref{eq:3-17} and \eqref{eq:3-22} that
\begin{align*}
&
\sum_{Q\in S_d}\mu(Q)\left(
\frac{1}{\mu(Q)}\int_Q\alpha_Q^{\lambda p(x)}\,d\mu(x)
\right)^{\frac{1}{\lambda}}
\\
&\quad\le 
\left\{
\left(\frac{2}{\eps}\right)^\phi\sum_{Q\in S_d}\int_Q\alpha_Q^{p(x)}\,d\mu(x)
+
\sum_{k=1}^\infty\left(\frac{2}{\eps}\right)^{\phi(k+1)}\beta^{p_-(k-1)}
\right\}^{\frac{1}{\lambda}}
\\
&\quad\le 
\left\{
\left(\frac{2}{\eps}\right)^\phi
+
\sum_{k=1}^\infty\left(\frac{2}{\eps}\right)^{\phi(k+1)}\beta^{p_-(k-1)}
\right\}^{\frac{1}{\lambda}}=:A.
\end{align*}
Combining $2/\eps>1$ and\eqref{eq:3-19}, we see that $A\in(1,\infty)$, 
which completes the proof of \eqref{eq:3-2}.
\end{proof}
\subsection{Second lemma}
The next lemma generalizes \cite[Lemma~5.2]{L17} with $w\equiv 1$ from the
Euclidean setting of $\R^n$ to the setting of spaces of homogeneous type.
\begin{lemma}\label{le:4}
Let $(X,d,\mu)$ be a space of homogeneous type and $p\in\cP(X)$ satisfy
$1<p_-,p_+<\infty$. Suppose the Hardy-Littlewood maximal operator $M$
is bounded on $L^{p(\cdot)}(X,d,\mu)$. There exist constants $B, \lambda>1$ 
and a measure $\nu$ on $X$ such that for every dyadic grid 
$\cD\in\bigcup_{t=1}^K\cD^t$ and every finite family of pairwise disjoint 
cubes $S_d\subset\cD$, the following properties hold:
\begin{enumerate}
\item[(i)] if $Q\in\cD$ and $t\ge 0$ satisfies
\begin{equation}\label{eq:4-1}
\int_Qt^{p(x)}\,d\mu(x)\le 1,
\end{equation}
then
\begin{equation}\label{eq:4-2}
\mu(Q)\left(
\frac{1}{\mu(Q)}\int_Q t^{\lambda p(x)}\,d\mu(x)
\right)^{\frac{1}{\lambda}}
\le 
B\int_Q t^{p(x)}\,d\mu(x)+\nu(Q);
\end{equation}
\item[(ii)]
\[
\sum_{Q\in S_d}\nu(Q)\le 2B.
\]
\end{enumerate}
\end{lemma}
\begin{proof}
Let $A,\lambda>1$ be the constants from Lemma~\ref{le:3}. Set
\begin{equation}\label{eq:4-3}
B:=2^{\frac{p_+}{p_-}+1}A.
\end{equation}
Fix a dyadic grid $\cD\in\bigcup_{t=1}^K\cD^t$. Given a cube $Q\in\cD$, 
consider the functions
\[
F_1(t):=\int_Qt^{p(x)}\,d\mu(x),
\quad
F_2(t):=\mu(Q)\left(
\frac{1}{\mu(Q)}\int_Q t^{\lambda p(x)}\,d\mu(x)
\right)^{\frac{1}{\lambda}},
\quad t\ge 0,
\]
and the set
\[
A(Q):=\{t>0\ : \ F_1(t)\le 1, \ F_2(t)>BF_1(t)\}.
\]
Set
\[
t_Q:=\left\{\begin{array}{lll}
0 &\mbox{if}& A(Q)=\emptyset,
\\
\sup A(Q) &\mbox{if}& A(Q)\ne\emptyset.
\end{array}\right.
\]

We claim that
\begin{equation}\label{eq:4-4}
F_1(t_Q)<1.
\end{equation}
Indeed, if $F_1(t_Q)=1$, then taking into account the continuity of $F_1$
and $F_2$, we would have $F_2(t_Q)\ge B>A$, and this would contradict
Lemma~\ref{le:3}.

Further, we also claim that
\begin{equation}\label{eq:4-5}
F_2(t_Q)=BF_1(t_Q).
\end{equation}
Indeed, otherwise we would have $F_2(t_Q)>BF_1(t_Q)$. This together with
\eqref{eq:4-4} and the continuity of the functions $F_1$ and $F_2$ would imply
that there exists $\eps>0$ such that
\[
F_1(t_Q+\eps)<1,
\quad
F_2(t_Q+\eps)>BF_1(t_Q+\eps),
\]
and these inequalities would contradict the definition of $t_Q$.

Set
\begin{equation}\label{eq:4-6}
\nu(Q):=F_2(t_Q)
\end{equation}
and suppose that \eqref{eq:4-1} is fulfilled. Since the function
$F_2$ is increasing, we see that
\begin{equation}\label{eq:4-7}
\mu(Q)\left(
\frac{1}{\mu(Q)}\int_Q t^{\lambda p(x)}\,d\mu(x)
\right)^{\frac{1}{\lambda}}
\le 
\nu(Q),
\quad
t\le t_Q.
\end{equation}
On the other hand, if $t>t_Q$, then $t\notin A(Q)$, whence
$F_2(t)\le BF_1(t)$, that is,
\begin{equation}\label{eq:4-8}
\mu(Q)\left(
\frac{1}{\mu(Q)}\int_Qt^{\lambda p(x)}\,d\mu(x)
\right)^{\frac{1}{\lambda}}
\le 
B\int_Q t^{p(x)}\,d\mu(x),
\quad t>t_Q.
\end{equation}
Combining \eqref{eq:4-7} and \eqref{eq:4-8}, we immediately arrive at
inequality \eqref{eq:4-2}, which means that property (i) is fulfilled.

Let us show that property (ii) is fulfilled. Consider an arbitrary
finite family $S_d\subset\cD$ of pairwise disjoint cubes. Among all
subsets $\widetilde{S}_d\subset S_d$ such that 
\begin{equation}\label{eq:4-9}
\sum_{Q\in\widetilde{S}_d}\int_Qt_Q^{p(x)}\,d\mu(x)\le 2,
\end{equation}
we choose a maximal subset $S_d'$ that is, a subset containing the 
largest number of cubes (it is not unique, in general).

We claim that $S_d'=S_d$. Indeed, assuming that $S_d'\subsetneqq S_d$ 
and taking into account that the function $F_1$ is increasing, we have
\[
\sum_{Q\in S_d'}\int_Q\left(\frac{t_Q}{2^{1/p_-}}\right)^{p(x)}\,d\mu(x)
\le 
\sum_{Q\in S_d'}\int_Q\frac{t_Q^{p(x)}}{2}\,d\mu(x)\le 1.
\]
By Lemma~\ref{le:3}, this inequality implies that
\[
\sum_{Q\in S_d'}\mu(Q)\left(
\frac{1}{\mu(Q)}\int_Q\left(\frac{t_Q}{2^{1/p_-}}\right)^{\lambda p(x)}\,d\mu(x)
\right)^{\frac{1}{\lambda}}\le A.
\]
Since the function $F_2$ is increasing, the above inequality yields
\[
\sum_{Q\in S_d'}\mu(Q)\left(\frac{1}{\mu(Q)}\int_Q 
\left(
\frac{1}{2^{p_+/p_-}}\right)^\lambda t_Q^{\lambda p(x)}\,d\mu(x)
\right)^{\frac{1}{\lambda}}
\le A, 
\]
whence
\[
\sum_{Q\in S_d'}\mu(Q)\left(
\frac{1}{\mu(Q)}\int_Q t_Q^{\lambda p(x)}\,d\mu(x)
\right)^{\frac{1}{\lambda}}
\le 2^{\frac{p_+}{p_-}}A.
\]
This inequality and equalities \eqref{eq:4-3} and \eqref{eq:4-5} imply that
\[
\sum_{Q\in S_d'}\int_Q t_Q^{p(x)}\,d\mu(x)
=
\frac{1}{B} \sum_{Q\in S_d'}\mu(Q)
\left(
\frac{1}{\mu(Q)}\int_Q t_Q^{\lambda p(x)}\,d\mu(x) 
\right)^{\frac{1}{\lambda}}
\le 
\frac{1}{2}.
\]
Now let $P\in S_d\setminus S_d'$. Then, taking into account \eqref{eq:4-4}, 
we get
\[
\sum_{Q\in S_d'\cup\{P\}}\int_Q t_Q^{p(x)}\,d\mu(x)
\le 
\frac{1}{2}+\int_P t_P^{p(x)}\,d\mu(x)
<\frac{3}{2}.
\]
This inequality in view of \eqref{eq:4-9} contradicts the maximality of $S_d'$.
This proves that $S_d'=S_d$. It follows from \eqref{eq:4-9} and 
\eqref{eq:4-5}--\eqref{eq:4-6} that
\[
\sum_{Q\in S_d}\nu(Q)
=
B\sum_{Q\in S_d}\int_Q t_Q^{p(x)}\,d\mu(x)
\le 2B,
\]
which completes the proof of property (ii) and the lemma.
\end{proof}
\subsection{Third lemma}
The next lemma is an extension of \cite[Lemma~5.3]{L17} with $w\equiv 1$ from
the Euclidean setting of $\R^n$ to the setting of spaces of homogeneous type.
\begin{lemma}\label{le:5}
Let $(X,d,\mu)$ be a space of homogeneous type and $p\in\cP(X)$ satisfy
$1<p_-,p_+<\infty$. Suppose the Hardy-Littlewood maximal operator $M$
is bounded on $L^{p(\cdot)}(X,d,\mu)$. There exist constants
$D,\gamma>1$ and $\zeta>0$ such that for every dyadic grid 
$\cD\in\bigcup_{t=1}^K\cD^t$ and every cube $Q\in\cD$, if
\begin{equation}\label{eq:5-1}
\min\left\{1,\frac{1}{\|\chi_Q\|_{L^{p(\cdot)}}^{1+\zeta}}\right\}
\le 
t
\le 
\max\left\{1,\frac{1}{\|\chi_Q\|_{L^{p(\cdot)}}^{1+\zeta}}\right\},
\end{equation}
then
\begin{equation}\label{eq:5-2}
\left(\frac{1}{\mu(Q)}\int_Q t^{\gamma p(x)}\,d\mu(x)\right)^{\frac{1}{\gamma}}
\le 
\frac{D}{\mu(Q)}\int_Q t^{p(x)}\,d\mu(x).
\end{equation}
\end{lemma}
\begin{proof}
Let $A>0$ and $\lambda>1$ be the constants of Lemma~\ref{le:3}. 
Take any
$\gamma$ satisfying $1<\gamma<\lambda$ and set
\[
\zeta:=\frac{\lambda-\gamma}{\gamma(1+\lambda)}>0.
\]
Fix a dyadic grid $\cD\in\bigcup_{t=1}^K\cD^t$ and a cube $Q\in\cD$.
For any $\alpha>0$, we have
\begin{equation}\label{eq:5-3}
\left(\frac{1}{\mu(Q)}\int_Q t^{\gamma p(x)}\,d\mu(x)\right)^{\frac{1}{\gamma}}
=
\left(
\frac{1}{\mu(Q)}\int_Q t^{\gamma(p(x)-\alpha)}\,d\mu(x)
\right)^{\frac{1}{\gamma}}t^\alpha.
\end{equation}
It follows from \eqref{eq:5-1} that either
\begin{equation}\label{eq:5-4}
1\le t\le\frac{1}{\|\chi_Q\|_{L^{p(\cdot)}}^{1+\zeta}},
\end{equation}
or
\begin{equation}\label{eq:5-5}
\frac{1}{\|\chi_Q\|_{L^{p(\cdot)}}^{1+\zeta}}
\le t \le 1.
\end{equation}
If \eqref{eq:5-4} is fulfilled and $\gamma(p(x)-\alpha)\ge 0$, then
\begin{equation}\label{eq:5-6}
t^{\gamma(p(x)-\alpha)}
\le 
\left(
\frac{1}{\|\chi_Q\|_{L^{p(\cdot)}}^{1+\zeta}}
\right)^{\gamma(p(x)-\alpha)}
<
1+\left(
\frac{1}{\|\chi_Q\|_{L^{p(\cdot)}}^{1+\zeta}}
\right)^{\gamma(p(x)-\alpha)}.
\end{equation}
On the other hand, if \eqref{eq:5-4} is fulfilled and $\gamma(p(x)-\alpha)<0$,
then
\begin{equation}\label{eq:5-7}
t^{\gamma(p(x)-\alpha)}\le 1<1+\left(
\frac{1}{\|\chi_Q\|_{L^{p(\cdot)}}^{1+\zeta}}
\right)^{\gamma(p(x)-\alpha)}.
\end{equation}
Analogously, if \eqref{eq:5-5} is fulfilled and $\gamma(p(x)-\alpha\ge 0$,
then inequality \eqref{eq:5-7} holds. On the other hand, if \eqref{eq:5-5}
is fulfilled and $\gamma(p(x)-\alpha)<0$, then inequality \eqref{eq:5-6}
holds.

It follows from above that if \eqref{eq:5-1} is fulfilled, then for all 
$x\in X$ and all $\alpha>0$,
\[
t^{\gamma(p(x)-\alpha)}
\le 
1+\|\chi_Q\|_{L^{p(\cdot)}}^{\gamma\alpha(1+\zeta)}
\left(\frac{1}{\|\chi_Q\|_{L^{p(\cdot)}}^{1+\zeta}}\right)^{\gamma p(x)}.
\]
Integrating this inequality over the cube $Q$ yields
\[
\int_Q t^{\gamma(p(x)-\alpha)}\,d\mu(x) 
\le 
\mu(Q)+\|\chi_Q\|_{L^{p(\cdot)}}^{\gamma\alpha(1+\zeta)}
\int_Q 
\left(\frac{1}{\|\chi_Q\|_{L^{p(\cdot)}}^{1+\zeta}}\right)^{\gamma p(x)}
\,d\mu(x).
\]
Hence, taking into account that $(a^\gamma+b^\gamma)^{1/\gamma}\le a+b$ for 
$a,b\ge 0$ and $\gamma>1$, we see that
\begin{align}\label{eq:5-8}
&
\left(
\frac{1}{\mu(Q)}\int_Q t^{\gamma(p(x)-\alpha)}\,d\mu(x)
\right)^{\frac{1}{\gamma}}
\\
&\quad\le 
\left(1+\left[
\|\chi_Q\|_{L^{p(\cdot)}}^{\alpha(1+\zeta)}
\left(
\frac{1}{\mu(Q)}\int_Q
\left(\frac{1}{\|\chi_Q\|_{L^{p(\cdot)}}^{1+\zeta}}\right)^{\gamma p(x)}
\,d\mu(x)
\right)^{\frac{1}{\gamma}}
\right]^\gamma\right)^{\frac{1}{\gamma}}
\nonumber\\
&\quad\le
1+\|\chi_Q\|_{L^{p(\cdot)}}^{\alpha(1+\zeta)}
\left(
\frac{1}{\mu(Q)}\int_Q
\left(\frac{1}{\|\chi_Q\|_{L^{p(\cdot)}}^{1+\zeta}}\right)^{\gamma p(x)}
\,d\mu(x)
\right)^{\frac{1}{\gamma}}.
\nonumber
\end{align}
Combining \eqref{eq:5-3} and \eqref{eq:5-8}, we obtain for $\alpha>0$,
\begin{align}\label{eq:5-9}
&
\left(\frac{1}{\mu(Q)}\int_Q t^{\gamma p(x)}\,d\mu(x)\right)^{\frac{1}{\gamma}}
\\
&\quad\quad\le 
t^\alpha+\|\chi_Q\|_{L^{p(\cdot)}}^{\alpha(1+\zeta)}
\left(\frac{1}{\mu(Q)}
\int_Q\left(\frac{1}{\|\chi_Q\|_{L^{p(\cdot)}}^{1+\zeta}}
\right)^{\gamma p(x)}\,d\mu(x)\right)^{\frac{1}{\gamma}}t^\alpha.
\nonumber
\end{align}
Let
\[
\alpha=m_p(Q)
\]
be a median value of the function $p$ over the cube $Q$, that is, a number 
satisfying
\[
\max\left\{
\frac{\mu\big(\{x\in Q: p(x)>m_p(Q)\}\big)}{\mu(Q)},
\frac{\mu\big(\{x\in Q: p(x)<m_p(Q)\}\big)}{\mu(Q)}
\right\}\le\frac{1}{2}.
\]
Set
\[
E_1(Q) := \{x\in Q\ :\ p(x)\le m_p(Q)\},
\quad
E_2(Q) := \{x\in Q\ :\ p(x)\ge m_p(Q)\}.
\]
It follows immediately from the definition of $m_p(Q)$ and these sets that
\begin{equation}\label{eq:5-10}
\mu(E_j(Q))\ge\frac{1}{2}\mu(Q),\quad j=1,2.
\end{equation}
Then, for $t\in(0,\infty)$, we have
\begin{align}\label{eq:5-11}
t^\alpha 
&=
t^{m_p(Q)}
\le \frac{2t^{m_p(Q)}\mu\big(E_j(Q)\big)}{\mu(Q)}
\\
&\le 
\left\{\begin{array}{lll}
\displaystyle
\frac{2}{\mu(Q)}\int_{E_1(Q)}t^{p(x)}\,d\mu(x) &\mbox{if}& t\in(0,1),
\\[5mm]
\displaystyle
\frac{2}{\mu(Q)}\int_{E_2(Q)}t^{p(x)}\,d\mu(x) &\mbox{if}& t\in[1,\infty)
\end{array}\right.
\nonumber\\
&\le\frac{2}{\mu(Q)}\int_Q t^{p(x)}\,d\mu(x).
\nonumber
\end{align}
We claim that
\begin{equation}\label{eq:5-12}
\chi_Q(x)\le 2(M^\cD\chi_{E_j(Q)})(x),
\quad x\in X,\quad j=1,2.
\end{equation}
Indeed, if $x\notin Q$, then the statement of \eqref{eq:5-12} is trivial.
On the other hand, if $x\in Q$, then \eqref{eq:5-10} implies that
\begin{align*}
\chi_Q(x)
&=
\frac{\mu(Q)}{\mu(Q)}\le\frac{2\mu\big(E_j(Q)\big)}{\mu(Q)}
=
\frac{2}{\mu(Q)}\int_Q\chi_{E_j(Q)}(y)\,d\mu(y)
\\
&\le 
2(M^\cD\chi_{E_j(Q)})(x),
\end{align*}
which completes the proof of \eqref{eq:5-12}.

It follows from \eqref{eq:5-12}, Theorem~\ref{th:HK-pointwise}, and the boundedness
of the Hardy-Littlewood maximal operator on $L^{p(\cdot)}(X,d,\mu)$ that
\begin{equation}\label{eq:5-13}
\|\chi_Q\|_{L^{p(\cdot)}}
\le 
2C_{HK}(X)\|M\|_{\cB(L^{p(\cdot)})}\|\chi_{E_j(Q)}\|_{L^{p(\cdot)}},
\quad j=1,2.
\end{equation}
In view of Lemma~\ref{le:3.2.5}, we have
\begin{align}\label{eq:5-14}
\|\chi_{E_j(Q)}\|_{L^{p(\cdot)}}
&\le 
\max\left\{
\big(\mu(E_j(Q))\big)^{1/p_-(E_j(Q))},
\big(\mu(E_j(Q))\big)^{1/p_+(E_j(Q))}
\right\}
\\
&\le 
\max\left\{
\big(\mu(Q)\big)^{1/p_-(E_j(Q))},
\big(\mu(Q)\big)^{1/p_+(E_j(Q))}
\right\},
\quad j=1,2.
\nonumber
\end{align}
Taking into account the definition of the sets $E_j(Q)$, we see that
\[
p_-(E_1(Q))\le p_+(E_1(Q))\le m_p(Q)\le p_-(E_2(Q))\le p_+(E_2(Q)).
\]
Therefore, if $\mu(Q)\le 1$, then
\begin{equation}\label{eq:5-15}
\max\left\{
\big(\mu(Q)\big)^{1/p_-(E_1(Q))},
\big(\mu(Q)\big)^{1/p_+(E_1(Q))}
\right\}
\le
\big(\mu(Q)\big)^{1/m_p(Q)},
\end{equation}
and if $\mu(Q)>1$, then
\begin{equation}\label{eq:5-16}
\max\left\{
\big(\mu(Q)\big)^{1/p_-(E_2(Q))},
\big(\mu(Q)\big)^{1/p_+(E_2(Q))}
\right\}
\le
\big(\mu(Q)\big)^{1/m_p(Q)} 
\end{equation}

If \eqref{eq:5-4} is fulfilled, then $\|\chi_Q\|_{L^{p(\cdot)}}\le 1$. Then,
in view of Lemma~\ref{le:3.2.4}, $\mu(Q)\le 1$. On the other hand if
\eqref{eq:5-5} is fulfilled, then $\|\chi_Q\|_{L^{p(\cdot)}}\ge 1$. Therefore,
by Lemma~\ref{le:3.2.4}, $\mu(Q)\ge 1$. Thus, if \eqref{eq:5-1} is 
fulfilled, then \eqref{eq:5-13}--\eqref{eq:5-16} imply that
\begin{equation}\label{eq:5-17}
\|\chi_Q\|_{L^{p(\cdot)}}
\le 
2C_{HK}(X)\|M\|_{\cB(L^{p(\cdot)}}\big(\mu(Q)\big)^{1/m_p(Q)}.
\end{equation}
Set
\[
q:=\frac{1+\lambda}{1+\gamma},
\quad
q':=\frac{q}{q-1}.
\]
Then
\begin{equation}\label{eq:5-18}
\gamma q(1+\zeta)
=
\gamma\frac{1+\lambda}{1+\gamma}
\left(1+\frac{\lambda-\gamma}{\gamma(1+\lambda)}\right)
=
\lambda
\end{equation}
and
\begin{equation}\label{eq:5-19}
\gamma q'\zeta
=
\gamma
\frac{1+\lambda}{1+\gamma}
\left(\frac{1+\lambda}{1+\gamma}-1\right)^{-1}
\frac{\lambda-\gamma}{\gamma(1+\lambda)}=1.
\end{equation}
Taking into account \eqref{eq:5-18}--\eqref{eq:5-19}, by H\"older's inequality
with the exponents $q,q'\in(1,\infty)$, we get
\begin{align}\label{eq:5-20}
&
\frac{1}{\mu(Q)}\int_Q\left(
\frac{1}{\|\chi_Q\|_{L^{p(\cdot)}}^{1+\zeta}}
\right)^{\gamma p(x)}\,d\mu(x)
\\
&\quad\quad\le 
\left(\frac{1}{\mu(Q)}\int_Q\left(
\frac{1}{\|\chi_Q\|_{L^{p(\cdot)}}}
\right)^{\gamma q(1+\zeta) p(x)}\,d\mu(x)\right)^{\frac{1}{q}}
\nonumber\\
&\quad\quad=
\left(\frac{1}{\mu(Q)}\int_Q\left(
\frac{1}{\|\chi_Q\|_{L^{p(\cdot)}}}\right)^{\lambda p(x)}\,d\mu(x)
\right)^{\frac{1}{q}}.
\nonumber
\end{align}
Since
\[
\int_Q\left(\frac{1}{\|\chi_Q\|_{L^{p(\cdot)}}}\right)^{p(x)}\,d\mu(x)
=
\int_X\left(\frac{\chi_Q(x)}{\|\chi_Q\|_{L^{p(\cdot)}}}\right)^{p(x)}\,d\mu(x)
\le 1,
\]
applying Lemma~\ref{le:3} with $S_d=\{Q\}$ and 
$\alpha_Q=1/\|\chi_Q\|_{L^{p(\cdot)}}$, we obtain
\[
\mu(Q)\left(
\frac{1}{\mu(Q)}\int_Q
\left(\frac{1}{\|\chi_Q\|_{L^{p(\cdot)}}}\right)^{\lambda p(x)}\,d\mu(x)
\right)^{\frac{1}{\lambda}} 
\le A.
\]
Hence
\begin{equation}\label{eq:5-21}
\left(\frac{1}{\mu(Q)}\int_Q
\left(\frac{1}{\|\chi_Q\|_{L^{p(\cdot)}}}\right)^{\lambda p(x)}\,d\mu(x)
\right)^{\frac{1}{q}}
\le 
\left(\frac{A}{\mu(Q)}\right)^{\frac{\lambda}{q}}.
\end{equation}
Combining \eqref{eq:5-20}--\eqref{eq:5-21}, we arrive at the following 
inequality:
\[
\frac{1}{\mu(Q)}\int_Q\left(
\frac{1}{\|\chi_Q\|_{L^{p(\cdot)}}^{1+\zeta}}
\right)^{\gamma p(x)}\, d\mu(x) 
\le 
\left(\frac{A}{\mu(Q)}\right)^{\frac{\lambda}{q}}.
\]
It follows from the previous estimate and estimate \eqref{eq:5-17} that
\begin{align}\label{eq:5-22}
&
\|\chi_Q\|_{L^{p(\cdot)}}^{\alpha(1+\zeta)} 
\left(\frac{1}{\mu(Q)}\int_Q\left( 
\frac{1}{\|\chi_Q\|_{L^{p(\cdot)}}^{1+\zeta}}
\right)^{\gamma p(x)}\,d\mu(x)\right)^{\frac{1}{\gamma}}t^\alpha
\\
&\quad\quad 
\le 
\left(2C_{HK}(X)\|M\|_{\cB(L^{p(\cdot)})}\right)^{\alpha(1+\zeta)}
(\mu(Q))^{\frac{\alpha(1+\zeta)}{m_p(Q)}}
\left(\frac{A}{\mu(Q)}\right)^{\frac{\lambda}{\gamma q}}t^\alpha
\nonumber\\
&\quad\quad 
=
\left(2C_{HK}(X)\|M\|_{\cB(L^{p(\cdot)})}\right)^{m_p(Q)(1+\zeta)} 
A^{\frac{\lambda}{\gamma q}}(\mu(Q))^{1+\zeta-\frac{\lambda}{\gamma q}}t^\alpha.
\nonumber
\end{align}
Taking into account the definitions of $\zeta$ and $q$, we see that
\begin{equation}\label{eq:5-23}
1+\zeta-\frac{\lambda}{\gamma q}
=
1+\frac{\lambda-\gamma}{\gamma(1+\lambda)}
-
\frac{1+\gamma}{1+\lambda}\cdot\frac{\lambda}{\gamma}=0.
\end{equation}
Combining \eqref{eq:5-9}, \eqref{eq:5-11} and \eqref{eq:5-22}--\eqref{eq:5-23},
we get
\[
\left(\frac{1}{\mu(Q)}\int_Qt^{\gamma p(x)}\,d\mu(x)\right)^{\frac{1}{\gamma}}
\le 
\frac{D}{2}t^{m_p(Q)} 
\le 
\frac{D}{\mu(Q)}\int_Q t^{p(x)}\,d\mu(x)
\]
with
\[
D:=1+\left(2C_{HK}(X)\|M\|_{\cB(L^{p(\cdot)})}\right)^{p_+(1+\zeta)}
A^{\frac{\lambda}{\gamma q}},
\]
which completes the proof of \eqref{eq:5-2}.
\end{proof}
\subsection{Fourth lemma}
The next lemma is an extension of \cite[Lemma~4.1]{L17} with $w\equiv 1$ from
the Euclidean setting of $\R^n$ to the setting of spaces of homogeneous type.
\begin{lemma}\label{le:6}
Let $(X,d,\mu)$ be a space of homogeneous type and $p\in\cP(X)$ satisfy
$1<p_-,p_+<\infty$. If the Hardy-Littlewood maximal operator $M$ is bounded
on $L^{p(\cdot)}(X,d,\mu)$, then there exist constants $C,\gamma>1$ and
$\eta>0$ and a measure $\nu$ on $X$ such that for every dyadic grid
$\cD\in\bigcup_{t=1}^K\cD^t$ and every finite family of pairwise disjoint 
cubes $S_d\subset\cD$ the following properties hold:
\begin{enumerate}
\item[(i)]
if $Q\in\cD$ and $t>0$ satisfies $t\|\chi_Q\|_{L^{p(\cdot)}}\le 1$, then
\begin{equation}\label{eq:6-1}
\mu(Q)\left(
\frac{1}{\mu(Q)}\int_Q t^{\gamma p(x)}\,d\mu(x)
\right)^{\frac{1}{\gamma}}
\le 
C\int_Q t^{p(x)}\,d\mu(x)+2t^\eta\nu(Q)\chi_{(0,1)}(t);
\end{equation}

\item[(ii)]
\[
\sum_{Q\in S_d}\nu(Q)\le C.
\]
\end{enumerate}
\end{lemma}
\begin{proof}
Let $B,D>1$, and $1<\gamma<\lambda$ be the constants given by
Lemmas~\ref{le:4} and~\ref{le:5} and $\nu$ be the measure on $X$ given by
Lemma~\ref{le:4}. Put
\[
\zeta:=\frac{\lambda-\gamma}{\gamma(1+\lambda)}
\]
and take
\begin{equation}\label{eq:6-2}
C:=\max\{(2B)^{1+\zeta},D\}.
\end{equation}
Then part (ii) follows from part (ii) of Lemma~\ref{le:4} because 
$C\ge 2B$.

Let us prove part (i). If $t\|\chi_Q\|_{L^{p(\cdot)}}\le 1$ and $t\ge 1$,
then $\|\chi_Q\|_{L^{p(\cdot)}}\le 1$ and, therefore, 
$\|\chi_Q\|_{L^{p(\cdot)}}^{1+\zeta}\le\|\chi_Q\|_{L^{p(\cdot)}}\le 1$.
It is easy to check that \eqref{eq:5-1} is fulfilled. Then it follows
from Lemma~\ref{le:5} and \eqref{eq:6-2} that
\[
\left(\frac{1}{\mu(Q)}\int_Q t^{\gamma p(x)}\,d\mu(x)\right)^{\frac{1}{\gamma}}
\le 
\frac{C}{\mu(Q)}\int_Q t^{p(x)}\,d\mu(x),
\]
which immediately implies \eqref{eq:6-1} and completes the proof of part (i)
for $t\ge 1$.

Assume that $t\|\chi_Q\|_{L^{p(\cdot)}}\le 1$ and $0<t<1$. If
\begin{equation}\label{eq:6-3}
\left(\frac{1}{\mu(Q)}\int_Q t^{\gamma p(x)}\,d\mu(x)\right)^{\frac{1}{\gamma}}
\le 
\frac{C}{\mu(Q)}\int_Q t^{p(x)}\,d\mu(x),
\end{equation}
then \eqref{eq:6-1} is trivial.

Assume that \eqref{eq:6-3} does not hold, that is,
\begin{equation}\label{eq:6-4}
\frac{1}{\mu(Q)}\int_Q t^{p(x)}\,d\mu(x)
< 
\frac{1}{C}\left(
\frac{1}{\mu(Q)}\int_Q t^{\gamma p(x)}\,d\mu(x)
\right)^{\frac{1}{\gamma}}.
\end{equation}
Take, as in the proof of Lemma~\ref{le:5},
\[
q:=\frac{1+\lambda}{1+\gamma},
\quad
q':=\frac{q}{q-1}.
\]
By H\"older's inequality, \eqref{eq:6-4} and \eqref{eq:5-18},
\begin{align}\label{eq:6-5}
\frac{1}{\mu(Q)}\int_Q t^{\frac{p(x)}{1+\zeta}}\,d\mu(x)
&\le 
\left(\frac{1}{\mu(Q)}\int_Q t^{p(x)}\,d\mu(x)\right)^{\frac{1}{1+\zeta}}
\\
&\le 
\left(\frac{1}{C}\right)^{\frac{1}{1+\zeta}}
\left(\frac{1}{\mu(Q)}\int_Q t^{\gamma q p(x)}\,d\mu(x)
\right)^{\frac{1}{\gamma q(1+\zeta)}}
\nonumber\\
&=
\left(\frac{1}{C}\right)^{\frac{1}{1+\zeta}}
\left(\frac{1}{\mu(Q)}\int_Q t^{\frac{\lambda p(x)}{1+\zeta}}\,d\mu(x)
\right)^{\frac{1}{\lambda}}.
\nonumber
\end{align}

It follows from \eqref{eq:6-2} and \eqref{eq:6-4} that
\[
\frac{D}{\mu(Q)}\int_Q t^{p(x)}\,d\mu(x)
\le 
\frac{C}{\mu(Q)}\int_Q t^{p(x)}\,d\mu(x) 
<
\left(
\frac{1}{\mu(Q)}\int_Q t^{\gamma p(x)}\,d\mu(x)
\right)^{\frac{1}{\gamma}},
\]
that is, \eqref{eq:5-2} does not hold. Therefore, by Lemma~\ref{le:5},
condition \eqref{eq:5-1} is not fulfilled. Since $0<t<1$, this means that
\[
0<t<\frac{1}{\|\chi_Q\|_{L^{p(\cdot)}}^{1+\zeta}},
\]
whence $\left\|t^{\frac{1}{1+\zeta}}\chi_Q\right\|_{L^{p(\cdot)}}\le 1$.
Therefore, by Lemma~\ref{le:3.2.4},
\[
\int_Q t^{\frac{p(x)}{1+\zeta}}\,d\mu(x)
=
\int_X\left(t^{\frac{1}{1+\zeta}}\chi_Q(x)\right)^{p(x)}\,d\mu(x)\le 1,
\]
that is, \eqref{eq:4-1} is fulfilled with $t^{\frac{1}{1+\zeta}}$ in place 
of $t$. The, by Lemma~\ref{le:4}, inequality \eqref{eq:4-2} holds with
$t^{\frac{1}{1+\zeta}}$ in place of $t$, that is,
\begin{equation}\label{eq:6-6}
\mu(Q)\left(
\frac{1}{\mu(Q)}\int_Q t^{\frac{\lambda p(x)}{1+\zeta}}\,d\mu(x)
\right)^{\frac{1}{\lambda}}
\le 
B\int_Q t^{\frac{p(x)}{1+\zeta}}\,d\mu(x)+\nu(Q).
\end{equation}
It follows from \eqref{eq:6-5}--\eqref{eq:6-6} and \eqref{eq:6-2} that
\begin{align*}
\mu(Q) & 
\left(\frac{1}{\mu(Q)}
\int_Q t^{\frac{\lambda p(x)}{1+\zeta}}\,d\mu(x)\right)^{\frac{1}{\lambda}}
\\
&\le 
B\left(\frac{1}{C}\right)^{\frac{1}{1+\zeta}}
\left(\frac{1}{\mu(Q)}\int_Q t^{\frac{\lambda p(x)}{1+\zeta}}\,d\mu(x)
\right)^{\frac{1}{\lambda}}+\nu(Q)
\\
&\le 
\frac{1}{2}
\left(\frac{1}{\mu(Q)}\int_Q t^{\frac{\lambda p(x)}{1+\zeta}}\,d\mu(x)
\right)^{\frac{1}{\lambda}}+\nu(Q).
\end{align*}
Thus
\begin{equation}\label{eq:6-7}
\left(\frac{1}{\mu(Q)}\int_Q t^{\frac{\lambda p(x)}{1+\zeta}}\,d\mu(x)
\right)^{\frac{1}{\lambda}}
\le
2\nu(Q).
\end{equation}
Since $0<t<1$, we have
\begin{equation}\label{eq:6-8}
t^{\frac{p(x)}{1+\zeta}}=
t^{p(x)}t^{-\frac{\zeta p(x)}{1+\zeta}}
\le 
t^{p(x)}t^{-\frac{\zeta p_- }{1+\zeta}}.
\end{equation}
Inequalities \eqref{eq:6-7} and \eqref{eq:6-8} imply that
\[
\mu(Q)\left(\frac{1}{\mu(Q)}t^{-\frac{\lambda\zeta p_- }{1+\zeta}}
\int_Q t^{\lambda p(x)}\,d\mu(x)\right)^{\frac{1}{\lambda}}
\le 2\nu(Q),
\]
whence 
\begin{equation}\label{eq:6-9}
\mu(Q)\left(
\frac{1}{Q}\int_Q t^{\lambda p(x)}\,d\mu(x)
\right)^{\frac{1}{\lambda}}
\le 2t^\eta\nu(Q)
\end{equation}
with
\[
\eta:=\frac{\zeta p_-}{1+\zeta}.
\]
Since $1<\gamma<\lambda$, by H\"older's inequality,
\begin{equation}\label{eq:6-10}
\left(
\frac{1}{\mu(Q)}\int_Q t^{\gamma p(x)}\,d\mu(x)
\right)^{\frac{1}{\gamma}}
\le 
\left(
\frac{1}{\mu(Q)}\int_Q t^{\lambda p(x)}\,d\mu(x)
\right)^{\frac{1}{\lambda}}.
\end{equation}
Combining \eqref{eq:6-9} and \eqref{eq:6-10}, we arrive at
\[
\mu(Q)\left(
\frac{1}{\mu(Q)}\int_Q t^{\gamma p(x)}\,d\mu(x)
\right)^{\frac{1}{\gamma}}
\le 2t^\eta\nu(Q),
\]
which implies \eqref{eq:6-1} and completes the proof of part (i)
for $0<t<1$.
\end{proof}
\section{Proof of Theorem~\ref{th:main}}
\label{sec:proof-main}
It is sufficient to show that if the Hardy-Littlewood maximal operator
$M$ is bounded on the space $L^{p(\cdot)}(X,d,\mu)$, then it is also
bounded on the space $L^{p'(\cdot)}(X,d,\mu)$. In turn, in view of
Corollary~\ref{co:RAE} it is enough to verify the condition $\cA_\infty$.
To achieve this aim, we will apply Lemma~\ref{le:2}. 

Let $\cD\in\bigcup_{t=1}^K\cD^t$ be a dyadic grid, $S\subset\cD$
be a finite sparse family, $\{G_Q\}_{Q\in S}$ be a collection of
nonnegative numbers such that
\[
\left\|\sum_{Q\in S}\alpha_Q\chi_Q\right\|_{L^{p(\cdot)}}=1.
\]
Then for every $Q\in S$,
\[
\alpha_Q\|\chi_Q\|_{L^{p(\cdot)}}
\le 
\left\|\sum_{Q\in S}\alpha_Q\chi_Q\right\|_{L^{p(\cdot)}}=1.
\]

Let $C,\gamma>1$, $\eta>0$ be the constants and $\nu$ be the measure from
Lemma~\ref{le:6}. Suppose $Q\in S$ is such that $\alpha_Q\ge 1$. Applying
H\"older's inequality and Lemma~\ref{le:6}, we get
\begin{align*}
\int_{G_Q}\alpha_Q^{p(x)}\,d\mu(x)
&=
\frac{\mu(Q)}{\mu(Q)}\int_Q\alpha_Q^{p(x)}\chi_{G_Q}(x)\,d\mu(x)
\\
&\le 
\left(\frac{\mu(G_Q)}{\mu(Q)}\right)^{\frac{1}{\gamma'}}
\mu(Q)\left(
\frac{1}{\mu(Q)}\int_Q\alpha_Q^{\gamma p(x)}\,d\mu(x)
\right)^{\frac{1}{\gamma}}
\\
&\le 
C\left(\frac{\mu(G_Q)}{\mu(Q)}\right)^{\frac{1}{\gamma'}}
\int_Q\alpha_Q^{p(x)}\,d\mu(x).
\end{align*}
Combining this inequality with Lemma~\ref{le:1}, we get
\begin{align}\label{eq:main-1}
\sum_{Q\in S\,:\,\alpha_Q\ge 1}\int_{G_Q}\alpha_Q^{p(x)}\,d\mu(x)
&\le 
C\sum_{Q\in S\,:\,\alpha_Q\ge 1}
\left(\frac{\mu(G_Q)}{\mu(Q)}\right)^{\frac{1}{\gamma'}}
\int_Q\alpha_Q^{p(x)}\,d\mu(x)
\\
&\le 
C\left(\max_{Q\in S}\frac{\mu(G_Q)}{\mu(Q)}\right)^{\frac{1}{\gamma'}}
\sum_{Q\in S}\int_Q\alpha_Q^{p(x)}\,d\mu(x)
\nonumber
\\
&\le 
C\left(\max_{Q\in S}\frac{\mu(G_Q)}{\mu(Q)}\right)^{\frac{1}{\gamma'}}.
\nonumber
\end{align}

For $k\in\N$, put
\begin{equation}\label{eq:main-2}
S_k:=\{Q\in S\ :\ 2^{-k}\le\alpha_Q <2^{-k+1}\}.
\end{equation}
If $S_k\ne\emptyset$, then there exists a number $i_k\in\N$
and cubes $Q_1^k,\dots, Q_{i_k}^k$ such that 
$\bigcup_{Q\in S_k}Q=\bigcup_{j=1}^{i_k} Q_j^k$; if 
$i,j\in\{1,\dots,i_k\}$ and $i\ne j$, then $Q_i^k\cap Q_j^k=\emptyset$;
and for every $Q\in S_k$, there is $j\in\{1,\dots,i_k\}$
such that $Q\subset Q_j^k$.

For $k\in\N$ and $S_k\ne\emptyset$, put
\begin{equation}\label{eq:main-3}
\psi_{Q_j^k}(x)=\sum_{Q\in S_k\,:\,Q\subset Q_j^k}\chi_{G_Q}(x),
\quad
j\in\{1,\dots,i_k\}.
\end{equation}
Then, taking into account \eqref{eq:main-2}, one has for all 
$j\in\{1,\dots,i_k\}$,
\begin{align}\label{eq:main-4}
\sum_{Q\in S_k\,:\,Q\subset Q_j^k}
\int_{G_Q}\alpha_Q^{p(x)}\,d\mu(x)
&= 
\sum_{Q\in S_k\,:\,Q\subset Q_j^k}
\int_{Q_j^k}\alpha_Q^{p(x)}\chi_{G_Q}(x)\,d\mu(x)
\\
&\le 
\sum_{Q\in S_k\,:\,Q\subset Q_j^k}
\int_{Q_j^k} (2^{-k+1})^{p(x)}\chi_{G_Q}(x)\,d\mu(x)
\nonumber\\
&=
\int_{Q_j^k}(2^{-k+1})^{p(x)}\psi_{Q_j^k}(x)\,d\mu(x)
\nonumber\\
&\le 
2^{p_+}\int_{Q_j^k}2^{-kp(x)}\psi_{Q_j^k}(x)\,d\mu(x)
\nonumber\\
&\le 
2^{p_+}\int_{Q_j^k}\alpha_{Q_j^k}^{p(x)}\psi_{Q_j^k}(x)\,d\mu(x).
\nonumber
\end{align}
By H\"older's inequality, for $k\in\N$ such that $S_k\ne\emptyset$ and 
$j\in\{1,\dots,i_k\}$, one has
\begin{align}\label{eq:main-5}
\int_{Q_j^k}\alpha_{Q_j^k}^{p(x)}\psi_{Q_j^k}(x)\,d\mu(x)
\le &
\mu(Q_j^k)
\left(\frac{1}{\mu(Q_j^k)}
\int_{Q_j^k}\alpha_{Q_j^k}^{\gamma p(x)}\,d\mu(x)
\right)^{\frac{1}{\gamma}}
\\
&\times
\left(\frac{1}{\mu(Q_j^k)}
\int_{Q_j^k}\psi_{Q_j^k}^{\gamma'}(x)\,d\mu(x)
\right)^{\frac{1}{\gamma'}}.
\nonumber
\end{align}
It follows from \eqref{eq:main-2} and the hypothesis that the sets
$\{G_Q\}_{Q\in S}$ are pairwise disjoint that
\begin{align}\label{eq:main-6}
\int_{Q_j^k}\psi_{Q_j^k}^{\gamma'}(x)\,d\mu(x)
&=
\sum_{Q\in S_k\,:\,Q\subset Q_j^k}\mu(G_Q)
\\
&\le 
\left(\max_{Q\in S}\frac{\mu(G_Q}{\mu(Q)}\right) 
\sum_{Q\in S_k\,:\, Q\subset Q_j^k}\mu(Q).
\nonumber
\end{align}
Since $S$ is a sparse family, there exists a collection of pairwise disjoint
sets $\{E(Q)\}_{Q\in S}$ such that $E(Q)\subset Q$ and $\mu(Q)\le2\mu(E(Q))$.
Hence, for all $k\in\N$ such that $S_k\ne\emptyset$ and all
$j\in\{1,\dots,i_k\}$,
\begin{align}\label{eq:main-7}
\sum_{Q\in S_k\,:\,Q\subset Q_j^k}\mu(Q) 
&\le 
2\sum_{Q\in S_k\,:\,Q\subset Q_j^k}\mu(E(Q))
\\
&=2\mu\left(\bigcup_{Q\in S_k\,:\,Q\subset Q_j^k}E(Q)\right)
\le 2\mu(Q_j^k).
\nonumber
\end{align}
On the other hand, taking into account that $\alpha_{Q_j^k}<1$, we deduce
from Lemma~\ref{le:6} that
\begin{equation}\label{eq:main-8}
\mu(Q_j^k)\left(\frac{1}{\mu(Q_j^k)}
\int_{Q_j^k}\alpha_{Q_j^k}\alpha_{Q_j^k}^{\gamma p(x)}\,d\mu(x)
\right)^{\frac{1}{\gamma}}
\le 
C\int_{Q_j^k}\alpha_{Q_j^k}^{p(x)}\,d\mu(x)
+
2\alpha_{Q_j^k}^\eta\nu(Q_j^k).
\end{equation}
Combining \eqref{eq:main-4}--\eqref{eq:main-8}, we obtain for every $k\in\N$
such that $S_k\ne\emptyset$ and every $j\in\{1,\dots,i_k\}$,
\begin{align*}
&
\sum_{Q\in S_k\,:\,Q\subset Q_j^k}\int_{G_Q}\alpha_Q^{p(x)}\,d\mu(x)
\\
&\quad\quad\quad\le
2^{p_++\frac{1}{\gamma'}}
\left(\frac{\mu(G_Q)}{\mu(Q)}\right)^{\frac{1}{\gamma'}}
\left(
C\int_{Q_j^k}\alpha_{Q_j^k}^{p(x)}\,d\mu(x)+2\alpha_{Q_j^k}^\eta\nu(Q_j^k)
\right).
\end{align*}
Then
\begin{align}\label{eq:main-9}
&
\sum_{Q\in S\,:\,\alpha_Q<1}\int_{G_Q}\alpha_Q^{p(x)}\,d\mu(x)
\\
&\quad\quad
=
\sum_{k\in\N\,:\, S_k\ne\emptyset}\sum_{Q\in S_k}
\int_{G_Q}\alpha_Q^{p(x)}\,d\mu(x)
\nonumber\\
&\quad\quad
=
\sum_{k\in\N\,:\, S_k\ne\emptyset}\sum_{j=1}^{i_k}
\sum_{Q\in S_k\,:\,Q\subset Q_j^k}
\int_{G_Q}\alpha_Q^{p(x)}\,d\mu(x)
\nonumber
\end{align}
\begin{align*}
&\quad\quad 
\le
2^{p_++\frac{1}{\gamma'}}
\left(\max_{Q\in S}\frac{\mu(G_Q)}{\mu(Q)}\right)^{\frac{1}{\gamma'}}
\nonumber\\
&\quad\quad\quad
\times
\sum_{k\in\N\,:\, S_k\ne\emptyset}\sum_{j=1}^{i_k}
\left(
C\int_{Q_j^k}\alpha_{Q_j^k}^{p(x)}\,d\mu(x)
+
2\alpha_{Q_j^k}^\eta\nu(Q_j^k)\right).
\nonumber
\end{align*}
It follows from Lemma~\ref{le:1} that
\begin{equation}\label{eq:main-10}
\sum_{k\in\N\,:\, S_k\ne\emptyset}\sum_{j=1}^{i_k}
\int_{Q_j^k}\alpha_{Q_j^k}^{p(x)}\,d\mu(x)
\le 
\sum_{Q\in S}\int_Q\alpha_Q^{p(x)}\,d\mu(x)\le 1.
\end{equation}
Since for every fixed $k$, the cubes $Q_1^k,\dots, Q_{i_k}^k$ are
pairwise disjoint, it follows from Lemma~\ref{le:6}(ii) and \eqref{eq:main-2}
that
\begin{align}\label{eq:main-11}
\sum_{k\in\N\,:\, S_k\ne\emptyset}\sum_{j=1}^{i_k}
\alpha_{Q_j^k}^\eta\nu(Q_j^k)
&\le 
\sum_{k\in\N\,:\, S_k\ne\emptyset}
(2^{-k+1})^\eta\sum_{j=1}^{i_k}\nu(Q_j^k)
\\
&\le 
C\sum_{k\in\N\,:\, S_k\ne\emptyset}2^{(-k+1)\eta}
\nonumber\\
&\le 
C2^\eta\sum_{k=1}^\infty\left(\frac{1}{2^\eta}\right)^k=:C_1.
\nonumber
\end{align}
It follows from \eqref{eq:main-9}--\eqref{eq:main-11} that
\begin{equation}\label{eq:main-12}
\sum_{Q\in S\,:\,\alpha_Q<1}\int_{G_Q}\alpha_Q^{p(x)}\,d\mu(x)
\le 
2^{p_++\frac{1}{\gamma'}}(C+C_1)
\left(\max_{Q\in S}\frac{\mu(G_Q)}{\mu(Q)}\right)^{\frac{1}{\gamma'}}.
\end{equation}
Combining \eqref{eq:main-1} and \eqref{eq:main-12}, we see that
\[
\sum_{Q\in S}\int_{G_Q}\alpha_Q^{p(x)}\,d\mu(x)
\le 
\Psi\left(\max_{Q\in S}\frac{\mu(G_Q)}{\mu(Q)}\right)^\xi
\]
with $\Psi:=C+2^{p_++\frac{1}{\gamma'}}(C+2C_1)$ and $\xi=1/\gamma'$. 
Hence \eqref{eq:2-1} implies \eqref{eq:2-2}.
By Lemma~\ref{le:2}, the space $L^{p(\cdot)}(X,d,\mu)$ satisfies the condition
$\cA_\infty$. Thus, the Hardy-Littlewood maximal operator $M$ is bounded
on the variable Lebesgue space $L^{p'(\cdot)}(X,d,\mu)$ in view of 
Corollary~\ref{co:RAE}.
\qed
\subsection*{Acknowledgments}
This work was partially supported by the Funda\c{c}\~ao para a Ci\^encia e a
Tecnologia (Portu\-guese Foundation for Science and Technology)
through the project UID/MAT/00297/2013 (Centro de Matem\'atica e 
Aplica\c{c}\~oes).


\end{document}